\documentclass[leqno,final]{siamltex}

\usepackage[final,margin]{fixme}
\usepackage{amsmath}
\usepackage{amsfonts}
\usepackage{amssymb}
\usepackage{bm}
\allowdisplaybreaks

\usepackage{graphicx}
\usepackage{caption}
\usepackage{subcaption}
\usepackage{afterpage}
\usepackage{multirow}
\usepackage{algorithm}
\usepackage{algpseudocode}
\floatname{algorithm}{Procedure}

\usepackage{hyperref}
\usepackage[format=hang, font={footnotesize,sf}, labelfont={bf}, margin=1cm, aboveskip=5pt,position=bottom]{caption}
\usepackage{multicol}
\newtheorem{remark}{Remark}

\usepackage{fancyhdr}

\usepackage[T1]{fontenc}
\usepackage[ansinew]{inputenc}
\usepackage{lmodern}

\usepackage[letterpaper,centering]{geometry}
\usepackage{todonotes}

\fxusetheme{color}
\FXRegisterAuthor{rv}{reviewer}{\color{green!60!black}RV}
\FXRegisterAuthor{ym}{aym}{YMM}
\FXRegisterAuthor{db}{adb}{\color{red}DB}

\usepackage{url}

\DeclareMathOperator*{\Span}{span}

\makeatletter
\newlength{\continueindent}
\renewenvironment{algorithmic}[1][0]%
   {%
   \edef\ALG@numberfreq{#1}%
   \def\@currentlabel{\theALG@line}%
   \setcounter{ALG@line}{0}%
   \setcounter{ALG@rem}{0}%
   \let\\\algbreak%
   \expandafter\edef\csname ALG@currentblock@\theALG@nested\endcsname{0}%
   \expandafter\let\csname ALG@currentlifetime@\theALG@nested\endcsname\relax%
   \begin{list}%
      {\ALG@step}%
      {%
      \rightmargin\z@%
      \itemsep\z@ \itemindent\z@ \listparindent2em%
      \partopsep\z@ \parskip\z@ \parsep\z@%
      \labelsep 0.5em \topsep 0.2em
      \ifthenelse{\equal{#1}{0}}%
         {\labelwidth 0.5em}%
         {\labelwidth 1.2em}%
       \leftmargin\labelwidth \addtolength{\leftmargin}{\labelsep}
      \ALG@tlm\z@%
      }%
      \parshape 2 \leftmargin \linewidth \continueindent \dimexpr\linewidth-\continueindent\relax
   \setcounter{ALG@nested}{0}%
   \ALG@beginalgorithmic%
   }%
   {
   \ALG@closeloops%
   \expandafter\ifnum\csname ALG@currentblock@\theALG@nested\endcsname=0\relax%
   \else%
      \PackageError{algorithmicx}{Some blocks are not closed!!!}{}%
   \fi%
   \ALG@endalgorithmic%
   \end{list}%
   }%
\makeatother

\title{Spectral tensor-train decomposition}

\author{Daniele Bigoni\footnotemark[1] \footnotemark[2]
  \and Allan P.\ Engsig-Karup\footnotemark[1]
  \and Youssef M.\ Marzouk\footnotemark[2]}

\begin{document}


\maketitle

\renewcommand{\thefootnote}{\fnsymbol{footnote}}
\footnotetext[1]{Technical University of Denmark, Kgs.\ Lyngby, DK-2800 Denmark. \texttt{\{dabi,apek\}@dtu.dk}}
\footnotetext[2]{Massachusetts Institute of Technology, Cambridge, MA 02139 USA. \texttt{\{dabi,ymarz\}@mit.edu}}
\renewcommand{\thefootnote}{\arabic{footnote}}

\begin{abstract} The accurate approximation of high-dimensional functions is an essential task in uncertainty quantification and many other fields. We propose a new function approximation scheme based on a spectral extension of the tensor-train (TT) decomposition. We first define a functional version of the TT decomposition and analyze its properties. We obtain results on the convergence of the decomposition, revealing links between the regularity of the function, the dimension of the input space, and the TT ranks. We also show that the regularity of the target function is preserved by the univariate functions (i.e., the ``cores'') comprising the functional TT decomposition. This result motivates an approximation scheme employing polynomial approximations of the cores. For functions with appropriate regularity, the resulting \textit{spectral tensor-train decomposition} combines the favorable dimension-scaling of the TT decomposition with the spectral convergence rate of polynomial approximations, yielding efficient and accurate surrogates for high-dimensional functions.
To construct these decompositions, we use the sampling algorithm \texttt{TT-DMRG-cross} to obtain the TT decomposition of tensors resulting from suitable discretizations of the target function.
We assess the performance of the method on a range of numerical examples: a modifed set of Genz functions with dimension up to $100$, and functions with mixed Fourier modes or with local features. We observe significant improvements in performance over an anisotropic adaptive Smolyak approach. The method is also used to approximate the solution of an elliptic PDE with random input data. The open source software and examples presented in this work are available online.\footnote{\url{http://pypi.python.org/pypi/TensorToolbox/}}

\end{abstract}



\pagestyle{myheadings}
\thispagestyle{plain}
\markboth{\MakeUppercase{D.\ Bigoni,  A.\ Engsig-Karup, Y.\ Marzouk}}{\MakeUppercase{Spectral tensor-train decomposition}}






\section{Introduction}

High-dimensional functions appear frequently in science and engineering applications, where a quantity of interest may depend in nontrivial ways on a large number of independent variables. In the field of uncertainty quantification (UQ), for example, stochastic partial differential equations (PDEs) are often characterized by hundreds or thousands of independent stochastic parameters. A numerical approximation of the PDE solution must capture the coupled effects of all these parameters on the entire solution field, or on any quantity of interest that is a functional of the solution field. Problems of this kind quickly become intractable when confronted with na\"{i}ve approximation methods, and the development of more effective methods is a long-standing challenge. This paper develops a new approach for high-dimensional function approximation, combining the discrete tensor-train format \cite{Oseledets2011} with spectral theory for polynomial approximation.

For simplicity, we will focus on real-valued functions representing the parameter dependence of a single quantity of interest. For a function $f \in L^2([a,b]^d)$, a straightforward approximation might involve projecting $f$ onto the space spanned by the tensor product of basis functions $\{\phi_{i_j}(x_j) \}_{i_j=1}^{n_j} \subset L^2([a,b])$ for $j=1 \ldots d$, obtaining:
\begin{equation}\label{eq:FullTensorApproach}
  f \simeq \sum_{i_1}^{n_1} \cdots \sum_{i_d}^{n_d} c_{i_1,\ldots,i_d} \left( \phi_{i_1} \otimes \cdots \otimes \phi_{i_d} \right) .
\end{equation}
This approach quickly becomes impractical as the parameter dimension $d$ increases, due to the exponential growth in the number of coefficients $c_{i_1, \ldots, i_d}$ and the computational effort (i.e., the number of function evaluations) required to determine their values. This growth is a symptom of the ``curse of dimensionality.''

Attempts to mitigate the curse of dimensionality typically employ some assumption about the structure of the function under consideration, effectively reducing the number of coefficients that must be computed. A widely successful class of methods involves interpolation or pseudospectral approximation with sparse grids \cite{Barthelmann2000,Xiu2005,Nobile2008,Constantine2012,Conrad2012}:
instead of taking a full tensor product approximation as in \eqref{eq:FullTensorApproach}, one considers a Smolyak sum \cite{Smolyak1963} of \textit{smaller} full-tensor approximations, each perhaps involving only a subset of the input parameters or at most low-order interactions among all the inputs. While the basis functions $\phi_i$ are typically selected \textit{a priori}, the components of the Smolyak sum can be chosen adaptively \cite{Conrad2012}. In general, these approaches work best when inputs to the target function $f$ are weakly coupled.

Other approaches to high-dimensional approximation rely on \textit{low-rank separated representations}, e.g., of the form:
\begin{equation}\label{eq:separated}
  f \simeq \sum_{i=1}^r c_i \gamma_{i,1} \otimes \cdots \otimes \gamma_{i,d} ,
\end{equation}
where the functions $\gamma_{i,1},\ldots,\gamma_{i,d}:[a,b] \rightarrow \mathbb{R}$, for $i=1 \ldots r$, are not specified \textit{a priori} and $r$ is ideally small (hence, the descriptor `low-rank'). In some cases, the chosen representation might separate only certain blocks of inputs to $f$, e.g., spatial and stochastic variables {\cite{Nouy2010,Tamellini2014,Chinesta2014}}. In general, however, inputs to $f$ can all be separated as in \eqref{eq:separated}. The representation in \eqref{eq:separated} is analogous to the \textit{canonical decomposition} of a tensor \cite{Kolda2009}. Many strategies for constructing low-rank separated representations of parameterized models have been developed \cite{Doostan2013,Litvinenko2013,Khoromskij2010,Khoromskij2011a,Nouy2009,Nouy2010,Tamellini2014,Chinesta2014,Matthies2012,Giraldi2013,Espig2014,Dolgov2014,Zhang2014}; these include the proper generalized decomposition \cite{Nouy2010,Tamellini2014,Chinesta2014}, least-squares approaches \cite{Doostan2009}, and tensor-structured Galerkin approximations \cite{Khoromskij2011a,Matthies2012,Espig2014}. Almost all of these approaches are ``intrusive'' in the sense that they require access to more than black-box evaluations of the target function $f$. But non-intrusive approaches have recently been developed as well \cite{Doostan2013}.

An alternative to the canonical tensor decomposition is the \textit{tensor-train} (TT) format for discrete tensors, introduced by \cite{Oseledets2011}. As we will describe in Section~\ref{sec:tensor-decomposition}, the TT format offers a number of advantages over the canonical decomposition, and it is therefore attractive to consider its application to function approximation. Recent work employing TT in the context of uncertainty quantification includes \cite{Litvinenko2013}, which uses the TT format to compress the operator and the polynomial coefficients arising in the stochastic Galerkin discretization of an elliptic PDE. In \cite{Khoromskij2010} the \textit{quantics tensor-train} (QTT) format is used to accelerate the preconditioned iterative solution of multiparametric elliptic PDEs. \cite{Zhang2014} uses TT-cross interpolation \cite{Oseledets2010} to evaluate the three-term recurrence relation used to find orthogonal polynomials and Gaussian quadrature points for arbitrary probability measures. \cite{Espig2014} compares the TT format with the canonical decomposition and the hierarchical Tucker decomposition, for the purpose of storing the operator derived from the Galerkin discretization of a stochastic PDE, and for computing the associated inner products.
While these efforts use the TT format to achieve important efficiency gains in solving particular UQ problems, they do not address the general non-intrusive function approximation problem considered in this paper. 


In this work, we will use classical polynomial approximation theory to extend the \textit{discrete} TT decomposition into a scheme for the approximation of continuous functions. To do this, we will first construct the \textit{functional} counterpart of the tensor-train decomposition and examine its convergence. We will prove that the functional TT decomposition converges for a wide class of functions in $L^2$ that satisfy a particular regularity condition; this result highlights connections between the regularity of the target function, the dimension of the input space, and the TT ranks. For this class of functions, we will also show that the weak differentiability of the target function is preserved by the univariate functions or ``cores'' comprising the functional TT decomposition, allowing us to apply polynomial approximation theory to the latter. The resulting combined \textit{spectral TT} approximation exploits the regularity of the target function $f$ and converges exponentially for smooth functions, but yields a representation whose complexity can scale linearly with dimension.


Other work in recent years has examined the connection between multivariate function decompositions and their discrete counterparts, represented by factorizations of matrices and tensors. A broad presentation of the functional analysis of Banach and Hilbert tensor spaces is presented in \cite{Hackbusch2012}. Some of these results are exploited in the construction of the functional tensor-train decomposition. Another building block for many aspects of our work is \cite{TOWNSEND2013}, which studies decompositions of bivariate functions and their connections to classical matrix factorizations. Moving from the bivariate to the general multivariate case, examples of tensor-format decompositions for particular functions are given in \cite{Tyrtyshnikov2003,Oseledets2012}. For functions in periodic Sobolev spaces, \cite{Schneider2014} develops results for the approximation rates of hierarchical tensor formats. Our work will provide related results for the tensor-train decomposition of functions on hypercubes equipped with a finite measure.

Moreover, we will focus on the non-intrusive setting where $f$ is a black-box function that can only be evaluated at chosen parameter values. Hence we must resort to a sampling method in constructing the spectral TT approximation: we will use the rank-revealing \texttt{TT-DMRG-cross} technique \cite{Savostyanov2011} to approximate the tensors resulting from suitable discretizations of $f$. We will then assess the performance of the spectral TT approximation on a range of target functions, including the Genz test functions and modifications thereof, functions with Fourier spectra chosen to illustrate particular challenges, functions with local features, and functions induced by the solution of a stochastic elliptic PDE. In all these examples, we will comment on the relationships between the degree of the polynomial approximation, the TT ranks, the accuracy of the overall approximation, and the scaling of computational effort with dimension.

The remainder of the paper is organized as follows. In Section \ref{sec:tensor-decomposition}, we recall the definitions and properties of several tensor decomposition formats, focusing on the TT decomposition. Section \ref{sec:approximation-theory} reviews relevant results on the approximation of functions in Sobolev spaces. In Section \ref{sec:spectr-tens-train}, we provide a constructive definition of the functional TT decomposition, discuss its convergence, and present results on the regularity of the decomposition. This leads to algorithms for constructing the spectral TT decomposition, whose practical implementations are summarized in Section \ref{sec:algorithm}. Section \ref{sec:numerical-examples} presents the numerical examples. Some technical results are deferred to the Appendix.


\section{Tensor decompositions}\label{sec:tensor-decomposition}

For the moment, assume that we can afford to evaluate the function $f:[a,b]^d \rightarrow \mathbb{R}$ at all points on a tensor grid $\bm{\mathcal{X}}=\times_{j=1}^d {\bf x}_j $, where ${\bf x}_j = (x_{j}^i)_{i=1}^{n_j}$ for $j=1,\ldots,d$ and $x_j^i \in [a,b] \subset \mathbb{R}$.  We denote $\bm{\mathcal{A}}(i_1,\ldots,i_d) = f(x_{i_1},\ldots,x_{i_d})$ and abbreviate the $d$-dimensional tensor by $\bm{\mathcal{A}} = f(\bm{\mathcal{X}})$.

In the special case of $d=2$, $\bm{\mathcal{A}}$ reduces to a matrix $\mathbf{A}$. The singular value decomposition (SVD) of this matrix,
\begin{equation}
  \mathbf{A} = {\bf U} {\boldsymbol \Sigma} {\bf V}^T,
\end{equation}
always exists and, since $\mathbf{A}$ is a real-valued matrix, is unique up to sign changes \cite{Trefethen1997}. The SVD can be used to obtain a low-rank approximation of $\mathbf{A}$ by truncating away the smallest singular values on the diagonal of $\boldsymbol \Sigma$ and the corresponding columns of ${\bf U}$ and ${\bf V}$. Unfortunately the SVD cannot be immediately generalized to tensors of dimension $d > 2$. Several approaches to this problem have been proposed over the years \cite{Kolda2009,Bro1998,Grasedyck2013}. Perhaps the most popular are the canonical decomposition (CANDECOMP) \cite{Carroll1970,Harshman1970}, the Tucker decomposition \cite{Tuck1963a}, and the tensor-train decomposition \cite{Oseledets2011}.

\subsection{Classical tensor decompositions}\label{sec:class-tens-decomp}
The {\it canonical decomposition} aims to represent $\bm{\mathcal{A}}$ as a sum of outer products:
\begin{equation}
  \bm{\mathcal{A}} \simeq \bm{\mathcal{A}}_{\text{CD}} = \sum_{i=1}^r \mathbf{A}_{i}^{(1)} \otimes \cdots \otimes \mathbf{A}_{i}^{(d)} \; ,
\end{equation}
where $\mathbf{A}_{i}^{(k)}$ is the $i$-th column of matrix $\mathbf{A}^{(k)} \in \mathbb{R}^{n_k \times r}$. The upper bound of summation $r$ is called the canonical rank of the tensor $\bm{\mathcal{A}}_{\text{CD}}$. The canonical decomposition is unique under mild conditions \cite{Sidiropoulos2000}. On the other hand a best rank-$r$ decomposition---where one truncates the expansion similarly to the SVD---does not always exist since the space of rank-$r$ tensors is not closed \cite{Kruskal1989,Silva2008}. Computation of the canonical decomposition based on the alternating least squares (ALS) method is not guaranteed to find a global minimum of the approximation error, and has a number of other drawbacks and corresponding workarounds \cite{Kolda2009}. 

The {\it Tucker decomposition} is defined as follows:
\begin{equation}
  \bm{\mathcal{A}} \simeq \sum_{i_1=1}^{r_1}\cdots \sum_{i_d=1}^{r_d} g_{i_1 \ldots  i_d} \left( \mathbf{A}_{i_1}^{(1)} \otimes \cdots \otimes \mathbf{A}_{i_d}^{(d)} \right) ,
\end{equation}
where the {\it core tensor} $\bm{\mathcal{G}}$, defined by $\bm{\mathcal{G}}(i_1, \ldots, i_d) = g_{i_1 \ldots  i_d}$, weighs interactions between different components in different dimensions. This expansion is not unique, due to the possibility of applying a rotation to the core tensor and its inverse to the components $\mathbf{A}^{(i)}$. However, the ability to recover a unique decomposition can be improved if some sparsity is imposed on the core tensor \cite{Martin2008}. 
The Tucker decomposition does not suffer from the same closure problem as the canonical decomposition,
but the number of parameters to be determined grows exponentially with the dimension $d$ due to the presence of the core tensor $\bm{\mathcal{G}}$. This cost limits the applicability of Tucker decomposition to relatively low-dimensional problems.

\subsection{Discrete tensor-train (DTT) decomposition}\label{subsec:TT-decomposition}
The dimension limitations of the Tucker decomposition can be overcome using a hierarchical singular value decomposition, where the tensor is not decomposed with a single core $\bm{\mathcal{G}} $ that simultaneously relates all the dimensions, but rather with a hierarchical tree of cores---usually binary---that relate a few dimensions at a time. This approach is called the hierarchical Tucker or $\mathcal{H}$-Tucker decomposition \cite{Hackbusch2009,Grasedyck2010}. 
A particular type of $\mathcal{H}$-Tucker decomposition is the tensor-train decomposition,
which retains many of the characteristics of the $\mathcal{H}$-Tucker decomposition but with a simplified formulation. (See \cite[Sec.~5.3]{Grasedyck2010} for a comparison.) The tensor-train decomposition has the following attractive properties:
\begin{itemize}
  \item existence of the full-rank approximation \cite[Thm.~2.1]{Oseledets2011},
  \item existence of the low-rank best approximation \cite[Cor.~2.4]{Oseledets2011},
  \item an algorithm that returns a quasi-optimal TT-approximation (see \eqref{eq:TT-SVD-error} and \cite[Cor.~2.4]{Oseledets2011}),
  \item memory complexity that scales linearly with dimension $d$ \cite[Sec.~3]{Oseledets2011},
  \item straightforward multi-linear algebra operations, and
  \item a sampling algorithm for constructing the TT-approximation, with a computational complexity that scales linearly with the dimension $d$ \cite{Savostyanov2011}.
\end{itemize}
\vspace{\baselineskip}

\begin{definition}[{Discrete tensor-train approximation}]\label{def:tens-train-approx} Let $\bm{\mathcal{A}} \in \mathbb{R}^{n_1\times\cdots\times n_d}$ have entries $\bm{\mathcal{A}}(i_1,\ldots,i_d)$. The TT-rank--${\bf r}=(r_0,\ldots,r_d)$ approximation of $\bm{\mathcal{A}}$ is $\bm{\mathcal{A}}_{TT} \in \mathbb{R}^{n_1\times\cdots\times n_d}$, defined as:
\begin{equation}\label{eq:Discr-TT-approximation}
  \begin{aligned}
    \bm{\mathcal{A}}(i_1,\ldots,i_d) &= \bm{\mathcal{A}}_{TT}(i_1,\ldots,i_d) + \bm{\mathcal{E}}_{TT}(i_1,\ldots,i_d) \\
    &= \sum_{\alpha_0,\ldots,\alpha_d=1}^{\bf r} G_1(\alpha_0,i_1,\alpha_1)\cdots G_d(\alpha_{d-1},i_d,\alpha_d) + \bm{\mathcal{E}}_{TT}(i_1,\ldots,i_d) \;,
  \end{aligned}
\end{equation}
where $\bm{\mathcal{E}}_{TT}$ is the residual term and $r_0=r_d=1$.
\end{definition}

\smallskip

The three-dimensional arrays $G_k(\alpha_k, i_k, \alpha_{k+1})$ are referred to as TT cores. The TT format approximates every entry of the tensor $\bm{\mathcal{A}}$ with a product of matrices, in particular with a sequence of  $r_k \times r_{k+1}$ matrices, each indexed by the parameter $i_{k+1}$. In other words, each core $G_k$ is ``connected'' to the adjacent cores $G_{k-1}$ and $G_{k+1}$ by summing over the indices $\alpha_{k-1}$ and $\alpha_k$; hence the name tensor `train.' It can be shown \cite{Oseledets2011} that there exists an exact TT representation ($\bm{\mathcal{E}}_{TT} = {\bf 0}$) for which
\begin{equation}\label{eq:TT-rank-unfolding}
  r_k =\rank\left(\mathbf{A}_k\right),  \qquad \forall k \in \{1,\ldots,d\},
\end{equation}
where $\mathbf{A}_k$ is the $k$-th unfolding of $\bm{\mathcal{A}}$, corresponding to the MATLAB/NumPy operation:
\begin{equation}\label{eq:tensor-unfolding}
  \mathbf{A}_k = {\rm reshape}\left( \bm{\mathcal{A}}, \prod_{s=1}^k n_s, \prod_{s=k+1}^d n_s \right).
\end{equation}
Furthermore, if $r_k \leq \rank(\mathbf{A}_k)$, a TT-rank--${\bf r}$ best approximation to $\bm{\mathcal{A}}$ in Frobenius norm, called $\bm{\mathcal{A}}^{\rm best}$, always exists, and the algorithm \texttt{TT-SVD} \cite{Oseledets2011} produces a quasi-optimal approximation to it. In particular, if $\bm{\mathcal{A}}_{TT}$ is the numerical approximation of $\bm{\mathcal{A}}$ obtained with \texttt{TT-SVD}, then
\begin{equation}\label{eq:TT-SVD-error}
  \Vert \bm{\mathcal{A}} - \bm{\mathcal{A}}_{TT} \Vert_F \leq \sqrt{d-1}\Vert \bm{\mathcal{A}} - \bm{\mathcal{A}}^{\rm best}\Vert_F \;.
\end{equation}
If the truncation tolerance for the SVD of each unfolding is set to $\delta = \varepsilon / \sqrt{d-1} \Vert \bm{\mathcal{A}} \Vert_F$, the \texttt{TT-SVD} is able to construct the approximation $\bm{\mathcal{A}}_{TT}$ such that
\begin{equation}
  \label{eq:TT-SVD-error-2}
  \Vert \bm{\mathcal{A}} - \bm{\mathcal{A}}_{TT} \Vert_F \leq \varepsilon \Vert \bm{\mathcal{A}} \Vert_F \;.
\end{equation}
Assuming that the TT-ranks are all equal, $r_k = r$, and that $n_k=n$, the TT-decomposition $\bm{\mathcal{A}}_{TT}$ requires the storage of $\mathcal{O}\left( dnr^2 \right) $ parameters. Thus the memory complexity of the representation \eqref{eq:Discr-TT-approximation} scales linearly with dimension. A further reduction in the required storage can be achieved using the \textit{quantics}-TT format \cite{Oseledets2010a,Khoromskij2011} which, for $n=2^m$, leads to $\mathcal{O}\left( dmr^2 \right)$ complexity.

The computational complexity of the \texttt{TT-SVD} depends on the selected accuracy, but for $r_k =r$ and $n_k=n$, the algorithm requires $\mathcal{O}\left( rn^d \right)$ flops. 
We see that this complexity grows exponentially with dimension and thus the \textit{curse of dimensionality} is not resolved, except for the memory complexity of the final compressed representation. 
At this stage, it is worth noting that using the tensor-train format rather than the more complex $\mathcal{H}$-Tucker decomposition relinquishes the possibility of implementing a parallel version of \texttt{TT-SVD} \cite{Grasedyck2010} and gaining a factor of $1/\log_2(d)$ in computational complexity. But this would still not resolve the exponential growth of computational complexity with respect to dimension. Another reason that the \texttt{TT-SVD} may not be immediately suitable for high-dimensional problems is that it first requires storage of the full tensor. This means that the initial memory requirements scale exponentially with the problem's dimension. In the next section we will discuss an alternative method for constructing a TT approximation of the tensor using a small number of function evaluations.

An open question in tensor-train decomposition regards the ordering of the $d$ indices of $\bm{\mathcal{A}}$; different orderings can lead to higher or lower TT-ranks, and change the memory efficiency of the representation accordingly. Given a particular permutation $\sigma$, we define the re-ordered tensor $\bm{\mathcal{B}}({\bf i}) = \bm{\mathcal{A}}(\sigma({\bf i}))$. One would like to find $\sigma$ such that the TT-ranks of $\bm{\mathcal{B}}$ are minimized. From \eqref{eq:TT-rank-unfolding} we see that the TT-ranks depend on the ranks of the unfoldings ${\bf B}_k$ of $\bm{\mathcal{B}}$, and from the definition of the unfolding \eqref{eq:tensor-unfolding} one sees that two indices $i$ and $j$ will influence the ranks of the matrices $\{ {\bf B}_k \}_{k=i}^{j-1}$. The permutation $\sigma$ should be chosen such that pairs of indices yielding high-rank unfoldings are contiguous, so that the rank will be high only for a limited number of unfoldings. If this does not happen, the non-separability of pairs of dimensions is carried from core to core, making the decomposition more expensive. Section \ref{sec:tens-train-proj-Fourier} will point to several examples where this problem arises.

\subsection{Cross-interpolation of tensors}\label{sec:tt-dmrg}
An alternative to \texttt{TT-SVD} is provided by the \texttt{TT-DMRG-cross} algorithm. (See \cite{Savostyanov2011} for a detailed description.) This method hinges on the notion of the density matrix renormalization group \cite{White1993} (DMRG) and on matrix skeleton decomposition \cite{Goreinov1997}. For $d=2$ and $\mathbf{A} \in \mathbb{R}^{m\times n}$, the skeleton decomposition is defined by:
\begin{equation}
  \mathbf{A} \simeq \mathbf{A}(:,\mathcal{J})\mathbf{A}(\mathcal{I},\mathcal{J})^{-1}\mathbf{A}(\mathcal{I},:) \;,
\end{equation}
where $\mathcal{I}=(i_1,\ldots,i_r)$ and $\mathcal{J}=(j_1,\ldots,j_r)$ are subsets of the index sets $[1,\ldots,m]$ and $[1,\ldots,n]$. The selection of the indices $(\mathcal{I},\mathcal{J})$ need to be such that most of the information contained in $\mathbf{A}$ is captured by the decomposition. It turns out that the optimal submatrix $\mathbf{A}(\mathcal{I},\mathcal{J})$ is that with maximal determinant in modulus among all the $r\times r$ submatrices of $\mathbf{A}$ \cite{Goreinov2010}. The problem of finding such a matrix is NP-hard \cite{Civril2009}. An approximation to the solution of this problem can be found using the {\tt maxvol} algorithm \cite{Goreinov2010}, in a row-column alternating fashion as explained in \cite{Oseledets2010}. Running \texttt{maxvol} is computationally inexpensive and requires $2c(n-r)r$ operations, where $c$ is usually a small constant in many practical applications.

The problem of finding the TT-decomposition $\bm{\mathcal{A}}_{TT}$ can be cast as the minimization problem
\begin{equation}\label{eq:TT-dmrg-minimization}
  \min_{G_1,\ldots,G_d} \Vert \bm{\mathcal{A}} - \bm{\mathcal{A}}_{TT} \Vert_F .
\end{equation}
One possible approach for solving this problem is \texttt{TT-cross} \cite{Oseledets2010}. Here the optimization is performed through \textit{left-to-right} and \textit{right-to-left} sweeps of the cores, using the matrix skeleton decomposition to find the most relevant fibers in the $d$ dimensional space. A \textit{fiber} is, for a $d$-dimensional tensor $\bm{\mathcal{A}}$, the equivalent of what rows and columns are for a matrix. In MATLAB notation, the $(i_1,\ldots,i_{k-1},i_{k+1},\ldots,i_d)$ fiber along the $k$-th dimension is $\bm{\mathcal{A}}(i_1,\ldots,i_{k-1},:,i_{k+1},\ldots,i_d)$. This approach provides linear scaling in the number of entries evaluated. On the other hand, it requires the TT-ranks to be known \textit{a priori} in order to select the correct number of fibers for each dimension. Underestimating these ranks leads to a poor (and in some cases erroneous) approximation, while overestimating them increases computational effort.

A more effective approach is the \texttt{TT-DMRG-cross} \cite{Savostyanov2011}, where the optimization is performed over two cores, $G_k$ and $G_{k+1}$, at a time. At step $k$ of the sweeps, the core $W_k(i_k,i_{k+1})=G_k(i_k)G_k(i_{k+1}) $ solving \eqref{eq:TT-dmrg-minimization} is found, and the cores $G_k$ and $G_{k+1}$ are recovered 
through the SVD of $W_k$. The relevant core $W_k$ is identified again using the maximum volume principle, by selecting the most important planes $\bm{\mathcal{A}}(i_1,\ldots,i_{k-1}, :, :, i_{k+2},\ldots,i_d)$ in the $d$-dimensional space. Unlike \texttt{TT-cross}, this method is \textit{rank-revealing}, meaning that the TT-ranks do not need to be guessed \textit{a priori}; instead, the method determines them automatically.

\section{Relevant results from approximation theory}\label{sec:approximation-theory}

The main objective of this work is to extend the TT format to functional approximations of $f$. To do this we need to consider the case where some smoothness can be assumed on $f$. Here we will review some concepts from polynomial approximation theory which, in subsequent sections, will be combined with the tensor-train decomposition.
In the following, we will make use of the Sobolev spaces:
\begin{equation}\label{eq:SobolevSpace}
  \mathcal{H}^k_\mu({\mathbf{I}}) = \left\lbrace f \in L^2_\mu({\mathbf{I}}) : \sum_{\vert {\mathbf{i}} \vert \leq k} \Vert D^{({\mathbf{i}})}f \Vert_{L^2_\mu({\mathbf{I}})} < +\infty \right\rbrace \;,
\end{equation}
where $k \geq 0$, $D^{({\mathbf{i}})}f$ is the ${\mathbf{i}}$-th weak derivative of $f$, $\mathbf{I} = I_1 \times \cdots \times I_d$ is a product of intervals of $\mathbb{R}$ and $\mu:\mathcal{B}({\mathbf{I}}) \rightarrow \mathbb{R}$ is a $\sigma$-finite measure on the Borel $\sigma$-algebra defined on ${\mathbf{I}}$. This space is equipped with the norm $\Vert \cdot \Vert^2_{\mathcal{H}^k_\mu({\mathbf{I}})}$ defined as
\begin{equation}
  \label{eq:SobolevNorm}
  \Vert f \Vert^2_{\mathcal{H}^k_\mu({\mathbf{I}})} = \sum_{\vert {\mathbf{i}} \vert \leq k} \Vert D^{({\mathbf{i}})}f \Vert^2_{L^2_\mu({\mathbf{I}})}
\end{equation}
and the semi-norm $\vert \cdot \vert_{{\mathbf{I}},\mu,k}$ given by
\begin{equation}\label{eq:SobolevSemiNorm}
  \vert f \vert^2_{{\mathbf{I}},\mu,k} = \sum_{\vert {\mathbf{i}} \vert = k} \Vert D^{({\mathbf{i}})}f \Vert^2_{L^2_\mu({\mathbf{I}})}.
\end{equation}
In the following we will assume that $\mu$ is a product measure satisfying $\mu({\bf I})=\prod_{i=1}^d \mu_i(I_i)$, 
where $\mu_i$ is a $\sigma$-finite measure on the Borel $\sigma$-algebra defined on $I_i$.

\subsection{Projection}\label{sec:projection}
A function $f \in L^2_{\mu}({\mathbf{I}})$ can be approximated by its projection onto a finite-dimensional subspace of $L^2_{\mu}({\mathbf{I}})$. The following results hold both for compact and non-compact supports ${\mathbf{I}}$.

\smallskip
%
%
\begin{definition}[Spectral expansion]\label{def:SpectralExpansion} Let $\mathbf{I} \subseteq \mathbb{R}^d$ and $f\in L_\mu^2({\mathbf{I}})$.
Let $\left\lbrace\Phi_{\mathbf{i}}\right\rbrace_{\vert \mathbf{i} \vert=0}^{\mathbf{\infty}}$ be a set of multivariate polynomials forming an orthonormal basis for $L_\mu^2({\mathbf{I}})$, where $\Phi_{\mathbf{i}}(\mathbf{x}) = \phi_{i_1,1}(x_1)\cdots \phi_{i_d,d}(x_d)$,  $\mathbf{i}=(i_1,\ldots,i_d)$, and $\phi_{i,j}$ is the degree-$i$ member of the family of univariate polynomials orthonormal with respect to the measure $\mu_j$. For $\mathbf{N}=(N_1,\ldots,N_d) \in \mathbb{N}_0^d$, the degree-$\mathbf{N}$ spectral expansion of $f$ is obtained from the projection operator $P_{\mathbf{N}}:L_\mu^2({\mathbf{I}})\rightarrow \Span(\left\lbrace\Phi_{\mathbf{i}}\right\rbrace_{\mathbf{i}=0}^{\mathbf{N}})$, where
\begin{equation}\label{eq:SpectralExpansion}
  P_{\mathbf{N}} f = \sum_{0 \leq \mathbf{i} \leq \mathbf{N}} c_{\mathbf{i}}\Phi_{\mathbf{i}} , \qquad c_{\mathbf{i}}=\int_{\mathbf{I}} f \, \Phi_{\mathbf{i}} d \mu(\mathbf{x}).
\end{equation}
and $\mathbf{i} \leq \mathbf{N}$ denotes $\bigwedge_{1 \leq j \leq d} \left({i}_j \leq {N}_j\right) $. The operator $P_{\mathbf{N}}$ is orthogonal in the inner product on $L^2_\mu({\mathbf{I}})$.
\end{definition}

\medskip

For simplicity, in the following we define $P_N := P_{\mathbf{N}}$ when $N_1 = \ldots = N_d = N$. The rate of convergence of the spectral expansion \eqref{eq:SpectralExpansion} is determined by the smoothness of $f$.
\smallskip
\begin{proposition}[Convergence of spectral expansion \cite{Hesthaven2008,Canuto2006}]\label{prop:SpectralConvergence} Let $f\in \mathcal{H}^k_\mu({\mathbf{I}})$ for $k\geq0$. Then
\begin{equation}
  \Vert f - P_Nf \Vert_{L_\mu^2({\mathbf{I}})} \leq C(k)N^{-k} \vert f \vert_{{\mathbf{I}},\mu,k} .
\end{equation}
\end{proposition}
 
In practice the coefficients $c_{\mathbf{i}}$ in \eqref{eq:SpectralExpansion} are approximated using discrete inner products based on quadrature rules of sufficient accuracy. We will focus here on $d$-dimensional quadrature rules produced by tensorizing univariate Gaussian rules---specifically, for dimension $i$, an $(N_i+1)$-point Gaussian quadrature rule \cite{Gautschi2004}. Let $\left (\mathbf{x}_{\mathbf{i}}, w_{\mathbf{i}} \right )_{{\mathbf{i}}=0}^{\mathbf{N}}$ be the points and weights describing such a rule \cite{Golub1969}.
A $d$-dimensional integral can then be approximated by:
\begin{equation}
  \int_{{\mathbf{I}}} f({\bf x}) d\mu({\bf x}) \approx \sum_{{\mathbf{i}}=0}^{\mathbf{N}} f(\mathbf{x}_{\mathbf{i}}) w_{\mathbf{i}} =: U_{\mathbf{N}} (f) .
\end{equation}
The discrete (and computable) version of the spectral expansion \eqref{eq:SpectralExpansion} is then defined as follows.

\begin{definition}[Discrete projection] Let $({\mathbf{x}}_{\mathbf{j}},w_{\mathbf{j}} )_{{\mathbf{j}}=0}^{\bf N}$ be a set of quadrature points and weights. The discrete projection of $f$ is obtained by the action of the operator $\widetilde{P}_{\mathbf{N}}:L_\mu^2({\mathbf{I}}) \rightarrow \Span(\left\lbrace\Phi_{\mathbf{i}}\right\rbrace_{\mathbf{i}=0}^{\mathbf{N}})$, defined as
\begin{equation}
  \label{eq:DiscreteSpectralExpansion}
  \widetilde{P}_{\mathbf{N}} f = \sum_{\mathbf{i}=0}^{\mathbf{N}} \tilde{c}_{\mathbf{i}}\Phi_{\mathbf{i}} , \qquad \tilde{c}_{\mathbf{i}}= U_{\mathbf{N}}(f \Phi_{\mathbf{i}}) = \sum_{\mathbf{i}=0}^{\mathbf{N}} f({\mathbf{x}}_{\mathbf{j}}) \Phi_{\mathbf{i}}({\mathbf{x}}_{\mathbf{j}}) w_{\mathbf{j}} .
\end{equation}
The operator $\widetilde{P}_{\mathbf{N}}$ is orthogonal on $L_\mu^2({\mathbf{I}})$ only for $\mathbf{N}\rightarrow\infty$.
\end{definition}

\medskip

\noindent This approximation to the orthogonal projection onto $\mathbb{P}_{{\mathbf{N}}}$, the space of polynomials of degree up to ${\mathbf{N}}$, is sometimes called a \textit{pseudospectral} approximation. For simplicity, we have focused on the fully tensorized case and tied the number of quadrature points to the polynomial degree. When the quadrature $U_{\mathbf{N}}$ is a Gauss rule, then the discrete projection is exact for $f\in \mathbb{P}_{{\mathbf{N}}}$. For any $f$, using a quadrature rule that is exact for polynomials up to degree $2 \mathbf{N}$ ensures that potential $\mathcal{O}(1)$ \textit{internal} aliasing errors in \eqref{eq:DiscreteSpectralExpansion} are avoided \cite{Conrad2012}.


\subsection{Interpolation}\label{sec:interpolation}
A function $f$ can also be approximated using interpolation on a set of nodes and assuming a certain level of smoothness in between them. Here we will consider piecewise linear interpolation and polynomial interpolation on closed and bounded domains ${\mathbf{I}} = I_1 \times \cdots \times I_d $. Other interpolation rules could be used inside the same framework for specific problems.

The linear interpolation of a function $f:[a,b]\rightarrow \mathbb{R}$ can be written in terms of basis functions called hat functions: given a set of distinct ordered nodes $\{x_i\}_{i=0}^N \in [a,b]$ with $x_0 = a$ and $x_N = b$, the hat functions are:
\begin{equation}\label{eq:multilin-shapefunc}
  e_i(x) = 
  \begin{cases}
    \frac{x-x_{i-1}}{x_i - x_{i-1}} & \text{if} \; x_{i-1} \leq x \leq x_i \wedge x \geq a\\
    \frac{x-x_{i+1}}{x_{i}-x_{i+1}} & \text{if} \; x_{i} \leq x \leq x_{i+1} \wedge x \leq b\\
    0 & \rm{otherwise}
  \end{cases} .
\end{equation}
When dealing with multiple dimensions, several options are available. A common choice of basis functions have their support over simplices around a node. This allows the basis functions $e_{\mathbf{i}}$ to remain linear. In this paper, we will instead use basis functions supported on the hypercubes adjacent to a node. These basis functions $e_{\mathbf{i}}$ can no longer be linear while preserving the linear interpolation property: they need to be bilinear in two dimensions, trilinear in three dimensions, and so on.
Letting $V$ be the set of piecewise continuous functions on ${\mathbf{I}}$, the multi-linear interpolation operator $I_{\mathbf{N}}:V \rightarrow \mathcal{C}^0({\mathbf{I}}) $ is then defined by
\begin{equation}\label{eq:multilin-interpolation}
  I_{\mathbf{N}} f({\bf x}) = \sum_{{\mathbf{i}}=0}^{\mathbf{N}} \hat{c}_{\mathbf{i}} e_{\mathbf{i}}({\bf x}), \qquad \hat{c}_{\mathbf{i}} = f\left( {\bf x}_{\mathbf{i}} \right),
\end{equation}
where $\{ {\bf x}_{\mathbf{i}} \}_{{\mathbf{i}}=0}^{\mathbf{N}} = \{ x_i^1 \}_{i=0}^{N_1} \times \cdots \times \{x_i^d \}_{i=0}^{N_d}$ is a tensor grid of points. Again we will use the notation $I_N := I_{\mathbf{N}}$ when $N_1 = \ldots = N_d=N$. The multi-linear interpolation (\ref{eq:multilin-interpolation}) is a projection, but in general not an orthogonal projection, on $L^2_\mu({\bf I})$. If the grid points are uniformly distributed, the convergence of this approximation is as follows.
\smallskip
\begin{proposition}[Convergence of linear interpolation \cite{Brenner2008}]
  Let $f \in \mathcal{H}^2_\mu({\mathbf{I}})$. Then
  \begin{equation}
    \Vert f - I_N f \Vert_{L^2_\mu({\mathbf{I}})} \leq C N^{-2} \vert f \vert_{{\mathbf{I}},\mu,2} .  
  \end{equation}
\end{proposition}
\medskip

The second type of interpolation we will use in this paper is Lagrange polynomial interpolation. It is based on the Lagrange polynomials $\{ l_i \}_{i=1}^N$, defined in the univariate case as
\begin{equation}
  \label{eq:LagrangePolynomials}
  l_i(x) = \prod_{\substack{0 \leq m < k \\ m\neq i}} \frac{x-x_m}{x_i-x_m},
\end{equation}
where the nodes $\{ x_i \}_{i=1}^k \in [a,b]$ are typically distributed non-uniformly over the interval; an example is the Gauss nodes used in Section \ref{sec:projection}. This choice is designed to avoid the \textit{Runge phenomenon} and hence assure a more accurate approximation. The univariate polynomial interpolation operator $\Pi_N: V \rightarrow {\rm span}\left( \{ l_i \}_{i=0}^{N} \right)$ is given by
\begin{equation}
  \label{eq:LagrangeInterpolation}
  \Pi_{N} f(x) = \sum_{i=0}^N \hat{c}_{i} l_{i}(x), \qquad \hat{c}_{i} = f\left( {x}_{i} \right).
\end{equation}
The polynomial interpolation $\Pi_{N}$ is also a projection, but in general not orthogonal on $L^2_\mu({\bf I})$.
Lagrange interpolation in the multivariate case presents many theoretical issues when used for interpolation on arbitrary nodes. In the scope of this paper, however, we will only consider tensor grids of nodes, for which the theory follows easily from the univariate case. As we will see in the next section, we will never explicitly construct these tensor grids, thanks to the TT decomposition and cross-interpolation. But the convergence properties of Lagrange interpolation on tensor grids will nonetheless be useful for analysis purposes.
The convergence of the Lagrange interpolant is again dictated by the smoothness of the function being approximated.
\smallskip
\begin{proposition}[Convergence of Lagrange interpolation \cite{Bernardi1992,Canuto2006}]
  Let $f \in \mathcal{H}^k_\mu({\mathbf{I}})$ for $k \geq 1$. Then
  \begin{equation}
    \Vert f - \Pi_N f \Vert_{L^2_\mu({\mathbf{I}})} \leq C(k) N^{-k} \vert f \vert_{{\mathbf{I}},\mu,k} .
  \end{equation}
\end{proposition}

Recall that Lagrange interpolation on $N+1$ Gauss nodes is equivalent to the degree-$N$ pseudospectral approximation (discrete projection) computed with the same nodes \cite{boyd2001chebyshev}; this equivalence also extends to the tensorized case.
%


\section{Spectral tensor-train decomposition}\label{sec:spectr-tens-train}

Now we blend the discrete tensor-train decomposition of Section \ref{subsec:TT-decomposition} with the polynomial approximations described in Section \ref{sec:approximation-theory}. First, we construct a continuous version of the tensor-train decomposition, termed the \textit{functional} tensor-train (FTT) decomposition. The construction proceeds by recursively decomposing non-symmetric square integrable kernels through auxiliary symmetric square integrable kernels, as in \cite{Schmidt1907}. Next, we prove that this decomposition converges under certain regularity conditions, and that the cores of the FTT decomposition inherit the regularity of the original function, and thus are amenable to spectral approximation when the original function is smooth. Based on this analysis, we propose an efficient approach to high-dimensional function approximation that employs only \textit{one-dimensional} polynomial approximations of the cores of the FTT decomposition, and we analyze the convergence of these approximations.

\subsection{Functional tensor-train decomposition}
Let $X \times Y \subseteq \mathbb{R}^d$ and let $ f $ be a Hilbert-Schmidt kernel with respect to the finite measure $\mu:\mathcal{B}(X \times Y)\rightarrow \mathbb{R} $, i.e., $f\in L_{\mu}^2(X\times Y)$. We restrict our attention to product measures, so $\mu = \mu_x \times \mu_y $. The operator
\begin{equation}\label{eq:IntegralOperatorBase}
  \begin{aligned}
    T : L_{\mu_{y}}^2(Y) &\rightarrow L_{\mu_x}^2(X)\\
    g &\mapsto \int_{Y} f(x,y) g(y) d\mu_y(y)
  \end{aligned} 
\end{equation}
is linear, bounded and compact \cite[Cor.~4.6]{halmos1978bounded}.
The Hilbert adjoint operator of $T$ is $T^{\ast}:L_{\mu_x}^2(X) \rightarrow L_{\mu_{y}}^2(Y)$.
Then $TT^{\ast}: L_{\mu_x}^2(X) \rightarrow L_{\mu_x}^2(X)$ is a compact Hermitian operator. By the spectral theory of compact operators, the spectrum of $TT^{\ast}$ comprises a countable set of eigenvalues whose only point of accumulation is zero \cite[Thm~8.3-1,8.6-4]{Kreyszig2007}. Since $TT^{\ast}$ is self-adjoint, its eigenfunctions $ \left\lbrace \gamma(x;i) \right\rbrace_{i=1}^\infty \subset L_{\mu_x}^2(X) $ form an orthonormal basis \cite[Cor.~4.7]{halmos1978bounded}.
The operator $T^{\ast}T: L_{\mu_{y}}^2(Y) \rightarrow L_{\mu_{y}}^2(Y)$ is also self-adjoint and compact, with eigenfunctions $ \left\lbrace \varphi(i;y) \right\rbrace_{i=1}^\infty \subset L_{\mu_{y}}^2(Y)$ and the same eigenvalues as $TT^{\ast}$. Then we have the following expansion of $f$.
\smallskip
\begin{definition}[{\rm Schmidt decomposition}]
Let the integral operators $TT^{\ast}$ and $T^{\ast}T$ have eigenvalues $\{ \lambda(i) \}_{i=1}^\infty$ and associated eigenfunctions $\left\lbrace \gamma(x;i) \right\rbrace_{i=1}^\infty$ and $\left\lbrace \varphi(i;y) \right\rbrace_{i=1}^\infty$, respectively. Then the Schmidt decomposition of $f$ is:
  \begin{equation}
    \label{eq:functional-SVD}
    f = \sum_{i=1}^\infty \sqrt{\lambda(i)} \gamma(\,\cdot\,;i) \otimes \varphi(i;\,\cdot\,) \;.
  \end{equation}
\end{definition}
In the general setting considered here, the convergence of \eqref{eq:functional-SVD} is in $L_\mu^2$.

\smallskip

Now  let $I_1 \times \cdots \times I_d = \mathbf{I} \subseteq \mathbb{R}^d$ and let $ f $ be a Hilbert-Schmidt kernel with respect to the finite measure $\mu:\mathcal{B}({\bf I})\rightarrow \mathbb{R} $, i.e., $f\in L_{\mu}^2({\bf I})$. We assume $\mu = \prod_{i=1}^d \mu_i $. Applying the Schmidt decomposition to $f$ with $X=I_1$ and $Y=I_2\times \cdots \times I_d$, we obtain
\begin{equation}\label{eq:FirstStepContDecomp}
  f({\bf x}) = \sum_{\alpha_{1}=1}^\infty \sqrt{\lambda_1(\alpha_1)} \, \gamma_1\left(x_1;\alpha_{1}\right) \varphi_1\left(\alpha_1;x_2,\ldots,x_{d}\right)  \;.
\end{equation}
Now proceed forward by letting $X=\mathbb{N} \times I_2$ and $Y=I_3\times \cdots \times I_d$, and let $\tau$ be the counting measure on $\mathbb{N}$. From the definition of counting measure, the orthonormality of $\varphi(\alpha_i;\,\cdot\,)$ for all $\alpha_i \in \mathbb{N}$, and the fact that $f \in L_\mu^2({\bf I})$, we have:
\begin{equation}
  \label{eq:HilbertSchmidtVarphi}
  \begin{aligned}
    &\int_{X\times Y} \left\vert \sqrt{\lambda_1(\alpha_1)}\varphi_1(\alpha_1;x_2,\ldots,x_d) \right\vert^2 d\tau(\alpha_1) d\mu_2(x_2) \cdots d\mu_d(x_d) = \\
    &\sum_{\alpha_1=1}^\infty \lambda_1(\alpha_1) \int_{I_2\times \cdots \times I_d} \left\vert \varphi_1(\alpha_1;x_2,\ldots,x_d) \right\vert^2 d\mu_2(x_2) \cdots d\mu_d(x_d) = \sum_{\alpha_1=1}^\infty \lambda_1(\alpha_1) < \infty \;.
  \end{aligned}
\end{equation}
This means that $ \left( \sqrt{\lambda_1}\varphi_1 \right) \in L_{\tau \times \mu_2 \times \cdots \times \mu_d}^2(X\times Y)$ and thus it is a Hilbert-Schmidt kernel. Then, using the Schmidt decomposition we obtain
\begin{equation}
  \label{eq:SVDSecondSubStep}
  \sqrt{\lambda_1(\alpha_1)}\varphi_1(\alpha_1;x_2,\ldots,x_d) = \sum_{\alpha_2=1}^\infty \sqrt{\lambda_2(\alpha_2)} \gamma_2(\alpha_1;x_2;\alpha_2) \varphi_2(\alpha_2;x_3,\ldots,x_d) \;.
\end{equation}
This expansion can now be substituted into \eqref{eq:FirstStepContDecomp}:
\begin{equation}
  \label{eq:SVDSecondStep}
   f({\bf x}) = \sum_{\alpha_{1}=1}^\infty \sum_{\alpha_2=1}^\infty \sqrt{\lambda_2(\alpha_2)} \gamma_1\left(x_1;\alpha_{1}\right) \gamma_2(\alpha_1;x_2;\alpha_2) \varphi_2(\alpha_2;x_3,\ldots,x_d) \;.
\end{equation}
Proceeding recursively one obtains
\begin{equation}
  \label{eq:ContTTDecomposition}
   f({\bf x}) = \sum_{\alpha_{1},\ldots,\alpha_{d-1}=1}^\infty \gamma_1\left(\alpha_0;x_1;\alpha_{1}\right) \gamma_2(\alpha_1;x_2;\alpha_2) \cdots \gamma_d\left(\alpha_{d-1};x_d;\alpha_{d}\right) \;,
\end{equation}
where $\alpha_0=\alpha_d=1$ and $\gamma_d\left(\alpha_{d-1};x_d;\alpha_{d}\right) := \sqrt{\lambda_{d-1}(\alpha_{d-1})} \varphi_d(\alpha_{d-1};x_d)$. We will call this format the \textit{functional} tensor-train (FTT) decomposition.

If we now truncate the FTT decomposition, we obtain the functional version of the tensor-train approximation.
\smallskip
\begin{definition}[{\rm FTT approximation}] Let $I_1 \times \cdots \times I_d = \mathbf{I} \subseteq \mathbb{R}^d$ and $f \in L_{\mu}^2({\bf I})$. For ${\bf r} = (1,r_1,\ldots,r_{d-1},1)$, a functional TT-rank--${\bf r}$ approximation of $f$ is:
\begin{equation}
  \label{eq:ContinuousTTapprox}
    f_{TT}({\bf x}) := \sum_{\alpha_0,\ldots,\alpha_d=1}^{\bf r} \gamma_1(\alpha_0,x_1,\alpha_1) \cdots \gamma_d(\alpha_{d-1},x_d,\alpha_d) \;,
\end{equation}
where $\gamma_i(\alpha_{i-1},\cdot,\alpha_{i}) \in L_{\mu_i}^2 $ and $ \left\langle\gamma_k(i,\cdot,m),\gamma_{k}(i,\cdot,n)\right\rangle_{L_{\mu_k}^2} = \delta_{mn} $. The residual of this approximation will be denoted by $R_{TT} := f-f_{TT}$. We will call $\{\gamma_i\}_{i=1}^d$ the cores of the approximation.
\end{definition}

\subsection{Convergence of the FTT approximation}
In this section we will investigate the convergecnce of \eqref{eq:ContinuousTTapprox} and in particular we will try to connect this convergence with the regularity of the approximated function $f$.
\smallskip
\begin{proposition}\label{prop:FTT-conv}
Let the functional tensor-train decomposition \eqref{eq:ContTTDecomposition} be truncated retaining the largest singular values $\{ \{ \sqrt{ \lambda_i (\alpha_i)}  \}_{\alpha_i=1}^{r_i} \}_{i=1}^d$. Then the residual of the approximation \eqref{eq:ContinuousTTapprox} fulfills the condition:
\rvnote*{\#1-1}{
\begin{equation}
  \label{eq:Cont-TT-convergence}
  \Vert R_{TT} \Vert^2_{L_{\mu}^2} = \min_{\substack{ g \in L_{\mu}^2 \\ \mathrm{TT-ranks}(g)={\bf r}}} \Vert f - g \Vert^2_{L_{\mu}^2} \leq \sum_{i=1}^{d-1} \sum_{\alpha_i=r_i +1}^\infty \lambda_i(\alpha_i)\;.
\end{equation}}
\end{proposition}
\begin{proof}
The first equality is due to the construction of $f_{TT}$ by a sequence of orthogonal projections that minimize the error in the $L^2_\mu$-norm. These projections are onto the subspaces spanned by the eigenfunctions of the Hermitian operators induced by the tensor $f$, and are thus optimal \cite{Smithies1937,Simsa1992}.

The error bound is obtained by induction. Below, and for the remainder of the proof, we omit the arguments of $\lambda_1$, $\gamma_1$, $\varphi_1$, etc.\ in order to simplify the notation. The first step of the decomposition \eqref{eq:FirstStepContDecomp} leads to:
\begin{equation}
  \label{eq:ErrorBoundFTT-step1}
  \begin{aligned}
    \Vert f - f_{TT} \Vert^2_{L_{\mu}^2} &= \left\Vert f - \sum_{\alpha_1=1}^{r_1} \sqrt{\lambda_1} \gamma_1 \varphi_1 + \sum_{\alpha_1=1}^{r_1} \sqrt{\lambda_1} \gamma_1 \varphi_1 -f_{TT} \right\Vert^2_{L_{\mu}^2} \\
    &= \sum_{\alpha_1=r_1+1}^{\infty} \lambda_1(\alpha_1) + \underbrace{\left\Vert \sum_{\alpha_1=1}^{r_1} \sqrt{\lambda_1} \gamma_1 \varphi_1 -f_{TT} \right\Vert^2_{L_{\mu}^2}}_{g_1(r_1)} \;,
  \end{aligned}
\end{equation}
where the second equality above is due to the following orthogonality
\begin{equation}
  \label{eq:ErrorBoundFTT-Ortho}
  \left\langle f - \sum_{\alpha_1=1}^{r_1} \sqrt{\lambda_1} \gamma_1 \varphi_1, \, f_{TT} - \sum_{\alpha_1=1}^{r_1} \sqrt{\lambda_1} \gamma_1 \varphi_1 \right\rangle_{L_{\mu}^2} = 0 \;,
\end{equation}
which follows from the orthogonality of $\{ \gamma_1(\alpha_1;\,\cdot\,) \}$ and of $\{ \varphi_1(\alpha_1;\,\cdot\,) \}$.
Next, let $\left( \sqrt{\lambda_1} \varphi_1 \right)(\alpha_1;x_2,\ldots,x_d) := \sqrt{\lambda_1(\alpha_1)} \varphi_1(\alpha_1;x_2,\ldots,x_d)$ and apply the second step of the decomposition to the last term of \eqref{eq:ErrorBoundFTT-step1}:
\rvnote*{\#1-1}{
\begin{equation}
  \label{eq:ErrorBoundFTT-step2}
  \begin{aligned}
    g_1(r_1) &= \left\Vert \sum_{\alpha_1=1}^{r_1} \gamma_1 ( \sqrt{\lambda_1} \varphi_1 ) - \sum_{\alpha_1=1}^{r_1}\sum_{\alpha_2=1}^{r_2} \gamma_1 \sqrt{\lambda_2} \gamma_2 \varphi_2 + 
      \sum_{\alpha_1=1}^{r_1}\sum_{\alpha_2=1}^{r_2} \gamma_1 \sqrt{\lambda_2} \gamma_2 \varphi_2 - f_{TT} \right\Vert^2_{L_{\mu}^2} \\
    & = \left\Vert \sum_{\alpha_1=1}^{r_1} \gamma_1 \sum_{\alpha_2=r_2+1}^{\infty} \sqrt{\lambda_2} \gamma_2 \varphi_2 \right\Vert^2_{L_{\mu}^2} + \left\Vert \sum_{\alpha_1=1}^{r_1}\sum_{\alpha_2=1}^{r_2} \gamma_1 \gamma_2 
( \sqrt{\lambda_2} \varphi_2 ) - f_{TT} \right\Vert^2_{L_{\mu}^2} \;. \\
  \end{aligned}
\end{equation}
The first term can be simplified as
\begin{equation}
  \label{eq:ErrorBoundFTT-step2-substep1}
  \begin{aligned}
    \left\Vert \sum_{\alpha_1=1}^{r_1} \gamma_1 \sum_{\alpha_2=r_2+1}^{\infty} \sqrt{\lambda_2} \gamma_2 \varphi_2 \right\Vert^2_{L_{\mu}^2} &= \sum_{\alpha_1=1}^{r_1} \sum_{\alpha_2=r_2+1}^{\infty} \lambda_2 \left\Vert \gamma_2(\alpha_1;\cdot;\alpha_2) \right\Vert^2_{L_{\mu_2}^2} \\
    &= \sum_{\alpha_2=r_2+1}^{\infty} \lambda_2 \left( \sum_{\alpha_1=1}^{\infty} \left\Vert \gamma_2 \right\Vert^2_{L_{\mu_2}^2} - \sum_{\alpha_1=r_1+1}^{\infty} \left\Vert \gamma_2 \right\Vert^2_{L_{\mu_2}^2} \right) \\
    &\leq \sum_{\alpha_2=r_2+1}^{\infty} \lambda_2 \;,
  \end{aligned}
\end{equation}
where the orthonormality property $\sum_{\alpha_1=1}^{\infty} \left\Vert \gamma_2(\alpha_1;\cdot;\alpha_2) \right\Vert^2_{L_{\mu_2}^2} = \left\Vert \gamma_2(\cdot;\cdot;\alpha_2) \right\Vert^2_{L_{\tau\times\mu_2}^2} = 1$ is used. Then
\begin{equation}
  \label{eq:ErrorBoundFTT-step2-substep2}
    g_1(r_1) \leq \sum_{\alpha_2=r_2+1}^{\infty} \lambda_2 + \left\Vert \sum_{\alpha_1=1}^{r_1}\sum_{\alpha_2=1}^{r_2} \gamma_1 \gamma_2 ( \sqrt{\lambda_2} \varphi_2 ) - f_{TT} \right\Vert^2_{L_{\mu}^2} \;. \\
\end{equation}
Plugging \eqref{eq:ErrorBoundFTT-step2-substep2} into \eqref{eq:ErrorBoundFTT-step1} and proceeding by induction, the bound \eqref{eq:Cont-TT-convergence} is obtained.}
\hfill \end{proof}

\medskip
The result given in Proposition \ref{prop:FTT-conv} does not directly involve any properties of the function $f$. Now we will try to link the error of the FTT approximation with the regularity of $f$. To do so, we will use the following auxiliary results: Proposition~\ref{prop:SchwabDecay}, which is a particular case of \cite[Prop.~2.21]{Schwab2006}, and Lemmas~\ref{lemma:SobolevSemiVsNorm}  and \ref{lemma:SobolevPhiVsF}, whose proofs are given in Appendix \ref{sec:sobol-cont-kern}.
\smallskip

\begin{proposition}\label{prop:SchwabDecay}
  Let ${\bf I} \subset \mathbb{R}^d$ be a bounded domain and $V\in L^2_{\mu \otimes \mu}({\bf I}\times {\bf I})$ be the symmetric kernel of the compact non-negative integral operator $\mathcal{V}:L^2_\mu({\bf I})\rightarrow L^2_\mu({\bf I})$. If $V \in \mathcal{H}^{k}_\mu({\bf I}\times {\bf I})$ with $k>0$ and $\{\lambda_m\}_{m\geq 1} $ denotes the eigenvalue sequence of $\mathcal{V}$, then
  \begin{equation}
    \label{eq:SchwabDecay}
    \lambda_m \leq \vert V \vert_{{\bf I}\times{\bf I},\mu,k} m^{-k/d} \qquad \forall m\geq 1 \;.
  \end{equation}
\end{proposition}

\begin{lemma}\label{lemma:SobolevSemiVsNorm}
  Let $f \in \mathcal{H}_\mu^k({\bf I})$, $\bar{\bf I}=I_2\times \cdots \times I_d$, and $J(x,\bar{x}) = \langle f(x,y), f(\bar{x},y) \rangle_{L^2_\mu(\bar{\bf I})} $. Then $J \in \mathcal{H}_\mu^k(I_1\times I_1)$ and
  \begin{equation}
    \label{eq:SobolevSemiVsNorm}
    \vert J \vert_{I_1\times I_1,\mu,k} \leq \Vert f \Vert^2_{\mathcal{H}_\mu^k({\bf I})} \;.
  \end{equation}
\end{lemma}

\begin{lemma}\label{lemma:SobolevPhiVsF}
  \rvnote*{\#1-1}{Let the function $\left( \sqrt{\lambda_i} \varphi_i \right) (\alpha_i;\,\cdot\,) \in \mathcal{H}_\mu^k(\tilde{\bf I})$ be H\"older continuous with exponent $\alpha > 1/2$, where $\tilde{\bf I}=I_{i+1}\times \cdots \times I_d$  is closed and bounded and $\bar{\bf I}=I_{i+2}\times \cdots \times I_d$.} Let
  \begin{equation}
    \label{eq:SobolevPhiVsFDef}
    \left( \sqrt{\lambda_i} \varphi_i \right)_{TT}(\alpha_i;\,\cdot\,)= \sum_{\alpha_{i+1}=1}^{r_{i+1}} \sqrt{\lambda_{i+1}(\alpha_{i+1})} \gamma_{i+1}(\alpha_i;x_{i+1};\alpha_{i+1}) \varphi_{i+1}(\alpha_{i+1};x_{i+2},\ldots,x_d)
  \end{equation}
  be the truncated Schmidt decomposition of $\left( \sqrt{\lambda_i} \varphi_i \right)(\alpha_i;\,\cdot\,)$. \rvnote*{\#1-1}{Then
  \begin{equation}
    \label{eq:SobolevPhiVsF}
    \sum_{\alpha_i=1}^{r_i} \left\Vert \left( \sqrt{\lambda_i} \varphi_i \right)(\alpha_i) \right\Vert^2_{\mathcal{H}_\mu^k(\tilde{\bf I})} \leq \left\Vert f \right\Vert^2_{\mathcal{H}_\mu^k({\bf I})} \;.
  \end{equation}}
\end{lemma}

\smallskip

For the sake of simplicity, in the following analysis we will let the ranks be ${\bf r}=(r,\ldots,r)$. Our main result, relating the regularity of $f$, the ranks $r$, and the input dimension $d$ to the error of the FTT approximation, is as follows.

\smallskip

\rvnote*{\#1-1}{
\begin{theorem}[Convergence of the FTT approximation]\label{thm:ftt-approx-conv}
  Let $f \in \mathcal{H}_\mu^k({\bf I})$ be a H\"older continuous function with exponent $\alpha > 1/2$ defined on the closed and bounded domain ${\bf I}$. Then
  \begin{equation}
    \label{eq:EigenvaluesDecay}
    \left\Vert R_{TT} \right\Vert^2_{L_\mu^2} \leq \Vert f \Vert^2_{\mathcal{H}_\mu^k({\bf I})} (d-1) \zeta(k,r+1) \qquad \text{for} \  r\geq 1 \;,
  \end{equation}
  where $\zeta$ is the Hurwitz zeta function. Furthermore
  \begin{equation}
    \label{eq:EigenvaluesDecay-conv}
    \lim_{r \rightarrow \infty} \left\Vert R_{TT} \right\Vert^2_{L_\mu^2} = 0 \qquad \text{for} \; k > 1 \;.
  \end{equation}
\end{theorem}}
\begin{proof}
  We start by considering the case ${\bf I} = I_1\times I_2 \times I_3$. Define the following approximations of $f$, using the Schmidt decomposition \eqref{eq:functional-SVD}:
  \begin{align}
    \label{eq:FTT-conv-FTTapprox}
    f_{TT,1} &= \sum_{\alpha_1=1}^{r_1} \sqrt{\lambda_1(\alpha_1)} \gamma_1(x_1;\alpha_1) \varphi_1(\alpha_1;x_2,x_3) \;, \\
    f_{TT} &= \sum_{\alpha_1=1}^{r_1} \gamma_1(x_1;\alpha_1) \left(\sqrt{\lambda_1}\varphi_1\right)_{TT}(\alpha_1;x_2,x_3) \;,
  \end{align}
  where
  \begin{equation}
    \label{eq:FTT-conv-FTTapprox-1}
    \left(\sqrt{\lambda_1}\varphi_1\right)_{TT}(\alpha_1;x_2,x_3) = \sum_{\alpha_2=1}^{r_2} \sqrt{\lambda(\alpha_2)} \gamma_2(\alpha_1;x_2;\alpha_2) \varphi_2(\alpha_2;x_3) \;.
  \end{equation}
 As in \eqref{eq:ErrorBoundFTT-Ortho}, $\left\langle f-f_{TT,1},f_{TT,1}-f_{TT}\right\rangle_{L^2_\mu({\bf I})}=0$ and hence
  \begin{equation}
    \label{eq:FTT-conv-FTTerror}
    \Vert R_{TT} \Vert^2_{L^2_\mu({\bf I})} = \Vert f - f_{TT} \Vert^2_{L^2_\mu({\bf I})} = \Vert f - f_{TT,1} \Vert^2_{L^2_\mu({\bf I})} + \Vert f_{TT,1} - f_{TT} \Vert^2_{L^2_\mu({\bf I})} \; .
  \end{equation}
  Exploiting the orthogonality of the singular functions, Proposition \ref{prop:SchwabDecay}, and Lemma \ref{lemma:SobolevSemiVsNorm}, we have
  \begin{equation}
    \label{eq:FTT-conv-FTTerror-1}
    \Vert f - f_{TT,1} \Vert^2_{L^2_\mu({\bf I})} = \sum_{\alpha_1=r_1+1}^\infty \lambda(\alpha_1) \leq \sum_{\alpha_1=r_1+1}^\infty \alpha_1^{-k} \vert J_0 \vert_k \leq \Vert f \Vert^2_{\mathcal{H}_\mu^k({\bf I})} \zeta(k, r_1 +1) \; ,
  \end{equation}
  where $J_0(x_1,\bar{x}_1) = \langle f(x_1,x_2,x_3), f(\bar{x}_1,x_2,x_3) \rangle_{L^2_\mu(I_2 \times I_3)}$. Similarly:
  \begin{equation}
    \label{eq:FTT-conv-FTTerror-2}
    \begin{aligned}
      \left\Vert \left( \sqrt{\lambda_1}\varphi_1 \right)(\alpha_1) - \left( \sqrt{\lambda_1}\varphi_1 \right)_{TT}(\alpha_1) \right\Vert^2_{L^2_\mu(I_2\times I_3)} &\leq \sum_{\alpha_2=r_2+1}^\infty \alpha_2^{-k} \vert J_1(\alpha_1) \vert_k \\
      &\leq \left\Vert \left( \sqrt{\lambda_1}\varphi_1 \right)(\alpha_1) \right\Vert^2_{\mathcal{H}_\mu^k(I_2 \times I_3)} \zeta(k, r_2 +1) \; ,
    \end{aligned}
  \end{equation}
  where $J_1(\alpha_1;x_2,\bar{x}_2) = \left\langle \left( \sqrt{\lambda_1}\varphi_1 \right)(\alpha_1;x_2,x_3), \left( \sqrt{\lambda_1}\varphi_1 \right)(\alpha_1;\bar{x}_2,x_3) \right\rangle_{L^2_\mu(I_3)}$. With the help of Lemma \ref{lemma:SobolevPhiVsF}, this leads to
  \rvnote*{\#1-1}{
  \begin{equation}
    \label{eq:FTT-conv-FTTerror-3}
    \begin{aligned}
      \Vert f_{TT,1} - f_{TT} \Vert^2_{L^2_\mu({\bf I})} &= \left\Vert \sum_{\alpha_1=1}^{r_1} \gamma_1  (\, \cdot \, ; \alpha_1) \left( \left( \sqrt{\lambda_1}\varphi_1 \right)  (\alpha_1;\,\cdot\,) - \left( \sqrt{\lambda_1}\varphi_1 \right)_{TT} (\alpha_1;\,\cdot\,) \right) \right\Vert^2_{L^2_\mu({\bf I})} \\
      &= \sum_{\alpha_1=1}^{r_1} \Vert \gamma_1  (\, \cdot \, ; \alpha_1) \Vert^2_{L^2_{\mu}(I_1)} \left\Vert \left( \sqrt{\lambda_1}\varphi_1 \right)  (\alpha_1;\,\cdot\,) - \left( \sqrt{\lambda_1}\varphi_1 \right)_{TT} (\alpha_1;\,\cdot\,) \right\Vert^2_{L^2_\mu(I_2\times I_3)} \\
      &\leq \sum_{\alpha_1=1}^{r_1} \left\Vert \left( \sqrt{\lambda_1}\varphi_1 \right)  (\alpha_1;\,\cdot\,) \right\Vert^2_{\mathcal{H}_\mu^k(I_2\times I_3)} \zeta(k, r_2 +1) \leq \Vert f \Vert^2_{\mathcal{H}_\mu^k({\bf I})} \zeta(k, r_2 +1) \;.
    \end{aligned}
  \end{equation}}
  Thus we obtain the bound
  \rvnote*{\#1-1}{
  \begin{equation}
    \label{eq:FTT-conv-3dbound}
    \Vert R_{TT} \Vert^2_{L^2_\mu({\bf I})} \leq \Vert f \Vert^2_{\mathcal{H}_\mu^k({\bf I})} \left[ \zeta(k, r_1 +1) + \zeta(k, r_2 +1)\right] \;.
  \end{equation}}
  Now let ${\bf I}=I_1\times \cdots \times I_d$ and ${\bf r}=(r,\ldots,r)$, for $r\geq 1$. Then
  \rvnote*{\#1-1}{
  \begin{equation}
    \label{eq:FTT-conv-ndbound}
    \begin{aligned}
      \Vert R_{TT} \Vert^2_{L^2_\mu({\bf I})} &\leq \Vert f \Vert^2_{\mathcal{H}_\mu^k({\bf I})} \sum_{i=1}^{d-1} \zeta(k, r_i +1) = \Vert f \Vert^2_{\mathcal{H}_\mu^k({\bf I})} (d-1) \zeta(k, r+1) \;.
    \end{aligned}
  \end{equation}}

  Let us now study the asymptotic behavior of $\left\Vert R_{TT} \right\Vert^2_{L_\mu^2}$ as $r\rightarrow \infty$. For $k>1$, we can use the bound:
  \begin{equation}
    \label{eq:FTT-conv-asymptotic-1}
    \zeta(k, r+1) = \sum_{i=r+1}^\infty i^{-k} \leq \int_{r+1}^\infty i^{-k} di = \frac{(r+1)^{-(k-1)}}{(k-1)} \;.
  \end{equation}
  Plugging this into \eqref{eq:FTT-conv-ndbound}, we obtain:
  \rvnote*{\#1-1}{
  \begin{equation}
    \label{eq:FTT-conv-asymptotic-2}
    \begin{aligned}
      \Vert R_{TT} \Vert^2_{L^2_\mu({\bf I})} &\leq \Vert f \Vert^2_{\mathcal{H}_\mu^k({\bf I})} (d-1) \frac{(r+1)^{-(k-1)}}{(k-1)} \;.
    \end{aligned}
  \end{equation}
  This leads to the asymptotic estimate \eqref{eq:EigenvaluesDecay-conv}, completing the proof.}
\hfill\end{proof}

\subsection{Regularity of the FTT decomposition}\label{sec:regul-ftt-decomp}
To construct polynomial approximations of the functional tensor-train decomposition, we would like this decomposition to retain the same regularity as the original function. In particular, in the scope of the polynomial approximation theory presented in Section \ref{sec:approximation-theory}, we need boundedness of the weak derivatives used to define the Sobolev spaces \eqref{eq:SobolevSpace}. With this perspective, we will require absolute convergence almost everywhere of the FTT decomposition. Smithies \cite[Thm.~14]{Smithies1937} proved that a kind of integrated H\"older continuity with exponent $\alpha > 1/2$ is a sufficient condition for the absolute convergence almost everywhere (a.e.) of the Schmidt decomposition. The condition required by Smithies is a generalization of H\"older continuity a.e.\ \cite{TOWNSEND2013}, as we show in Appendix \ref{sec:hold-cont-smith}. The Smithies result can be extended by construction to the FTT decomposition:

\begin{corollary}[Absolute convergence almost everywhere]
  Let $I_1 \times \cdots \times I_d = \mathbf{I} \subset \mathbb{R}^d$ be closed and bounded, and $f \in L_\mu^2({\bf I})$ be a H\"older continuous function with exponent $\alpha > 1/2$. Then the FTT decomposition \eqref{eq:ContTTDecomposition} converges absolutely almost everywhere.
\end{corollary}

\smallskip

Now we can prove that if $f$ belongs to a certain Sobolev space, then the cores of the FTT decomposition will also belong to the same Sobolev space.

\begin{theorem}[Sobolev regularity of FTT cores]\label{thm:ftt-sobolev}
Let $I_1 \times \cdots \times I_d = \mathbf{I} \subset \mathbb{R}^d$ be closed and bounded, and let $f \in L_\mu^2({\bf I})$ be a H\"older continuous function with exponent $\alpha > 1/2$ such that $f\in \mathcal{H}^k_\mu({\bf I})$. 
Then the FTT decomposition \eqref{eq:ContTTDecomposition} is such that 
$\gamma_j(\alpha_{j-1},\cdot,\alpha_j) \in \mathcal{H}^{k}_{\mu_j}(I_j) $ for all $j$, $\alpha_{j-1}$, and $\alpha_j$.
\end{theorem}
\begin{proof}
  We will first show this property for the Schmidt decomposition \eqref{eq:functional-SVD} of the H\"older ($\alpha>1/2$) continuous function $f\in \mathcal{H}^k_\mu(X\times Y)$. First we want to show that
  \begin{equation}\label{eq:WeakDerivativeDominatedConv}
    D^{\bf i}f = \sum_{j=1}^\infty \sqrt{\lambda_j} (D^{i_1} \psi_j \otimes D^{i_2}\phi_j) \;,
  \end{equation}
  where ${\bf i} = (i_1,i_2)$. Since $f$ is H\"older ($\alpha>1/2$) continuous, \eqref{eq:functional-SVD} converges absolutely a.e.\ by Smithies \cite{Smithies1937}; then we can define
  \begin{equation}
    \infty > g := \sum_{j=1}^\infty \left\vert \sqrt{\lambda_j} (\psi_j \otimes \phi_j) \right\vert \geq \left\vert \sum_{j=1}^\infty \sqrt{\lambda_j} (\psi_j \otimes \phi_j) \right\vert  \;,
  \end{equation}
  where the domination holds almost everywhere. 
  Letting $\mathcal{C}^\infty_c(X\times Y)$ be the set of infinitely differentiable functions with compact support, by the definition of the weak derivative, for all $v \in \mathcal{C}^\infty_c(X\times Y)$,
  \begin{equation}
    (-1)^{\vert {\bf i} \vert} \int_{X \times Y} D^{\bf i} f v d\mu = \int_{X \times Y} f v^{({\bf i})} d\mu \;.
  \end{equation}
Therefore this property also holds for any $v = v_x \otimes v_y \in \mathcal{C}^\infty_c(X) \otimes \mathcal{C}^\infty_c(X)$. Using the dominated convergence theorem, we obtain:
  \begin{equation}\nonumber
    \begin{aligned}
      (-1)^{\vert {\bf i} \vert} &\int_{X \times Y} D^{\bf i} f v d\mu = \int_{X \times Y} f v^{({\bf i})} d\mu = \int_{X \times Y} \left( \sum_{j=1}^\infty \sqrt{\lambda_j} (\psi_j \otimes \phi_j) \right) v^{({\bf i})} d\mu \\
      &= \sum_{j=1}^\infty \sqrt{\lambda_j} \int_{X \times Y} (\psi_j \otimes \phi_j) v^{({\bf i})} d\mu = \sum_{j=1}^\infty \sqrt{\lambda_j} \int_{X \times Y} \left(\psi_jv_x^{(i_1)}\right) \otimes \left(\phi_jv_y^{(i_2)}\right) d\mu \\
      &= \sum_{j=1}^\infty \sqrt{\lambda_j} \left( (-1)^{i_1} \int_{X} D^{i_1}\psi_j v_x d\mu_x \right) \left( (-1)^{i_2} \int_{Y} D^{i_2}\phi_j v_y d\mu_y \right) \;.
    \end{aligned}
  \end{equation}
  Thus \eqref{eq:WeakDerivativeDominatedConv} holds. Next we want to show that $ f \in \mathcal{H}^k_\mu(X \times Y)$ implies $ \Vert D^{i_1} \psi_j \Vert_{L^2_\mu(X)} < \infty$ and $ \Vert D^{i_2} \phi_j \Vert_{L^2_\mu(Y)} < \infty$ for $i_1,i_2 \leq k$. Thanks to \eqref{eq:WeakDerivativeDominatedConv} and due to the orthonormality of $\{\phi_j \}_{j=1}^\infty$, we have that
  \begin{equation}
    D^{i_1} \psi_j = \frac{1}{\sqrt{\lambda_j}} \left\langle D^{(i_1,0)}f, \phi_j \right\rangle_{L^2_\mu(Y)} \;.
  \end{equation}
  Using the Cauchy-Schwarz inequality:
  \begin{equation}
    \begin{aligned}
      \left\Vert D^{i_1} \psi_j \right\Vert^2_{L^2_\mu(X)} &= \left\Vert \frac{1}{\sqrt{\lambda_j}} \left\langle D^{(i_1,0)}f, \phi_j \right\rangle_{L^2_\mu(Y)} \right\Vert^2_{L^2_\mu(X)} \\
      &\leq \left\vert \frac{1}{\lambda_j} \right\vert \left\Vert \phi_j \right\Vert^2_{L^2_\mu(Y)} \left\Vert D^{(i_1,0)} f \right\Vert^2_{L^2_\mu(X\times Y)} < \infty \;,
    \end{aligned}
  \end{equation}
  where the last bound is due to the fact that $\{\phi_j \}_{j=1}^\infty \subset L^2_\mu(Y)$ (see \eqref{eq:IntegralOperatorBase} and \eqref{eq:FirstStepContDecomp}) and $ D^{(i_1,0)} f \in L^2_\mu(X\times Y)$ because $i_1 \leq k$ and $f\in \mathcal{H}^k_\mu(X\times Y)$. In the same way, $\left\Vert D^{i_2} \phi_j \right\Vert_{L^2_\mu(Y)} < \infty$ for all $i_2 \leq k$. It follows that $\{\psi_j\}_{j=1}^\infty \subset \mathcal{H}^k_\mu(X)$ and $\{\phi_j\}_{j=1}^\infty \subset \mathcal{H}^k_\mu(Y)$.

The extension to the FTT decomposition \eqref{eq:ContTTDecomposition} follows by induction. Letting $X=I_1$ and $Y=I_2\times \cdots \times I_d$, we have $\{\gamma(\cdot;\alpha_1)\}_{\alpha_1=1}^\infty \subset \mathcal{H}^k_\mu(I_1)$ and $\{\varphi_1(\alpha_1;\cdot)\}_{\alpha_1=1}^\infty \subset \mathcal{H}^k_\mu(I_2\times\cdots\times I_d)$. We can then apply the same argument to the Schmidt decomposition of $\{\varphi_1(\alpha_1;\cdot)\}_{\alpha_1=1}^\infty$ and to every other set $\{\varphi_i(\alpha_{i-1};\cdot;\alpha_i)\}_{\alpha_i=1}^\infty$ obtained during the recursive construction of the FTT decomposition.
\hfill \end{proof}

\medskip

\begin{remark}{\rm
The results above have the limitation of holding for functions defined on closed and bounded domains. In many practical cases, however, functions are defined on the real line, equipped with a finite measure. To the authors' knowledge, the corresponding result for such cases has not been established in the literature. The result by Smithies \cite[Thm.~14]{Smithies1937} hinges on a result by Hardy and Littlewood \cite[Thm.~10]{Hardy1927} on the convergence of Fourier series; the latter is the only step in \cite[Thm.~14]{Smithies1937} where the closedness and boundedness of the domain is explicitly used. A similar result for an orthogonal system in $L^2_{\mu}(-\infty,\infty)$, where $\mu$ is a finite measure, would be sufficient to extend Smithies' result to the real line. For one of the numerical examples presented later (Section \ref{sec:elliptic-equation}), we will assume that this result holds.
}\end{remark}

Other regularity properties can be proven, given different kinds of continuity of the function $f$. These properties are not strictly necessary in the scope of polynomial approximation theory, so we will state them without proof. The first regards the continuity of the cores of the FTT decomposition and follows directly from Mercer's theorem \cite{jorgens1982linear}.
\smallskip
\begin{proposition}[Continuity of FTT cores]\label{prop:continuity}
  Let $I_1 \times \cdots \times I_d = \mathbf{I} \subset \mathbb{R}^d$, and let $f \in L_{\mu}^2({\bf I})$ be a continuous function with FTT decomposition \eqref{eq:ContTTDecomposition}.
  Then the cores $\gamma_i(\alpha_{i-1},\cdot,\alpha_{i})$ are continuous for every $i$ and $\alpha_i$.
\end{proposition}

\smallskip

The second property regards the strong derivatives of the cores of the FTT decomposition. It requires the Lipschitz continuity of the function and then follows from a result on the uniform convergence of the Schmidt decomposition by Hammerstein \cite{Hammerstein1923,TOWNSEND2013}.
\smallskip

\begin{theorem}[Differentiability of FTT cores]\label{thm:regularity}
  Let $I_1 \times \cdots \times I_d = \mathbf{I} \subset \mathbb{R}^d$ be closed and bounded, and let $f \in L_\mu^2({\bf I})$ be a Lipschitz continuous function such that $\frac{\partial^\beta f}{\partial x_1^{\beta_1} \cdots \partial x_d^{\beta d}} $ exists and is continuous on $\bf I$ for $\beta = \sum_{i=1}^d \beta_i$. Then the FTT decomposition \eqref{eq:ContTTDecomposition} is such that $ \gamma_k(\alpha_{k-1},\cdot,\alpha_k) \in \mathcal{C}^{\beta_k}(I_k) $ for all $k$, $\alpha_{k-1}$, and $\alpha_k$.
\end{theorem}

\subsection{Connecting the DTT and FTT decompositions}
\label{sec:dttftt}
The practical construction of the FTT decomposition must rely on evaluations of the function $f$ at selected points in its domain. It is natural to describe these pointwise evaluations through a discrete TT decomposition. The construction of the discrete TT decomposition, whether through \texttt{TT-SVD}, \texttt{TT-cross}, or \texttt{TT-DMRG-cross}, is based on the nonlinear minimization problem \eqref{eq:TT-dmrg-minimization}, leading to the approximation error \eqref{eq:TT-SVD-error-2} defined in terms of the Frobenius norm. The FTT decomposition requires instead solving the analogous minimization problem \eqref{eq:Cont-TT-convergence} defined in terms of the functional $L_\mu^2$ norm. We must then find a connection between these two minimization problems.

Using the fact that $\mu$ is a product measure, one can construct the tensor-product quadrature rule $Q$ defined by the points and weights $(\bm{\mathcal{X}},\bm{\mathcal{W}})$, where $\bm{\mathcal{X}}=\times_{j=1}^d {\bf x}_j$, $\bm{\mathcal{W}} = {\bf w}_1 \otimes \cdots \otimes {\bf w}_{d_{\rm s}}$, and $({\bf x}_j,{\bf w}_j)$ defines a Gauss-type quadrature rule in the $j$th dimension with respect to the measure $\mu_j$; see Section \ref{sec:approximation-theory}. Now let $h(\bm{\mathcal{X}}_{\bf i}) = f(\bm{\mathcal{X}}_{\bf i}) \sqrt{\bm{\mathcal{W}}_{\bf i}}$, where ${\bf i}=(i_1,\ldots,i_d)$. Then, for $\bm{\mathcal{B}} = h(\bm{\mathcal{X}})$,
\begin{equation}
  \label{eq:P1:Ch5:STT:ReWeight}
  \Vert f \Vert_{L^2_{\mu}} = \sum_{i_1 = 1}^{n_1} \cdots \sum_{i_d = 1}^{n_d} f^2(\bm{\mathcal{X}}_{\bf i}) \bm{\mathcal{W}}_{\bf i} + \mathcal{O}(N^{-k}) = \Vert \bm{\mathcal{B}} \Vert_F^2 + \mathcal{O}(N^{-k}) \;,
\end{equation}
where the approximation is exact for polynomial functions up to order $2 n_j-1$. One can then seek the DTT decomposition $\bm{\mathcal{B}}_{TT}$ satisfying 
$$
\Vert \bm{\mathcal{B}} - \bm{\mathcal{B}}_{TT} \Vert_F \leq \varepsilon \Vert \bm{\mathcal{B}} \Vert_F 
$$
using one of the methods outlined in Section \ref{sec:tensor-decomposition}. 
This approach allows us to approximate the solution of the minimization problem \eqref{eq:Cont-TT-convergence}, achieving a relative error
\begin{equation}
  \label{eq:ApproximationErrorFTT}
  \Vert f - f_{TT} \Vert_{L_\mu^2} \lesssim \varepsilon \Vert f \Vert_{L_\mu^2} \;.
\end{equation}
The error in the approximation bound \eqref{eq:ApproximationErrorFTT} is due to truncation error introduced by replacing the $L^2_\mu$ norm with a finite order quadrature rule, as well as aliasing due to the approximation of $f(\bm{\mathcal{X}}_{\bf i}) \sqrt{\bm{\mathcal{W}}_{\bf i}}$ by $\bm{\mathcal{B}}_{TT}({\bf i})$. Both of these errors disappear as $n_1,\ldots,n_d$ are increased. An appropriately error-weighted DTT decomposition of $\bm{\mathcal{A}}=f(\bm{\mathcal{X}})$ can then be recovered as $\bm{\mathcal{A}}^w_{TT} = \bm{\mathcal{B}}_{TT}/\sqrt{\bm{\mathcal{W}}}$, where we assume strictly positive quadrature weights. The numerical tests presented in Section \ref{sec:numerical-examples} confirm the idea that the relative $L^2_\mu$ error shown in \eqref{eq:ApproximationErrorFTT} can be achieved for sufficiently large $n_i$.

Note that the approach just described is not limited to Gaussian quadrature rules. For instance, with a uniform measure $\mu$ one could use a Newton-Cotes rule---e.g., a trapezoidal rule with equally spaced points and uniform weights---to approximate the $L^2_\mu$ norm. In this case, $\bm{\mathcal{B}}  = h(\bm{\mathcal{X}}) \propto f(\bm{\mathcal{X}}) = \bm{\mathcal{A}}$, and the DTT approximation can be applied directly to $\bm{\mathcal{A}}$.

\subsection{Polynomial approximation of the FTT decomposition}
All the theory needed to combine the FTT decomposition with the polynomial approximations described in Section \ref{sec:approximation-theory} is now in place. We will consider the projection and  interpolation approaches separately.

\subsubsection{Functional tensor-train projection}\label{sec:tens-train-proj}
Let $f \in \mathcal{H}^{k}_\mu(\mathbf{I})$ and let $f_{TT}$ be the rank--$\mathbf{r}$ FTT approximation of $f$. Applying the projector \eqref{eq:SpectralExpansion} to $f_{TT}$ yields $P_{\mathbf{N}} f_{TT} = \sum_{{\bf i}=0}^{\mathbf{N}} \tilde{c}_{\bf i} \Phi_{\bf i}$, where
\begin{equation}\label{eq:TensorTrainProjection}
  \tilde{c}_{\bf i} = \int_{\bf I} f_{TT}(\mathbf{x}) \Phi_{\bf i}(\mathbf{x}) \, d\mu(\mathbf{x}) = \sum_{\alpha_0,\ldots,\alpha_d = 1}^{\bf r}\beta_1(\alpha_0,i_1,\alpha_1) \cdots \beta_d(\alpha_{d-1},i_d,\alpha_d)
\end{equation}
and
\begin{equation}
    \beta_n(\alpha_{n-1},i_n,\alpha_{n}) = \int_{I_n}\gamma_n(\alpha_{n-1},x_n,\alpha_n) \, \phi_{i_n}(x_n) \, d\mu_n(x_n) .
\end{equation}
The spectral expansion of $f_{TT}$ can thus be obtained by projecting its cores $\gamma_n(\alpha_{n-1}, x_n, \alpha_n)$ onto univariate basis functions. Furthermore, we immediately have, via \eqref{eq:TensorTrainProjection}, a tensor-train representation of the expansion coefficients $\bm{\mathcal{C}}:=[c_{\bf i}]_{{\bf i} = 0}^N $.

By Theorems \ref{thm:ftt-approx-conv} and \ref{thm:ftt-sobolev}, the convergence of the spectral expansion depends on the regularity of $f$. Let $f\in \mathcal{H}_\mu^k({\bf I})$ for $k>d-1$. Then:
\rvnote*{\#1-1}{
\begin{equation}
  \begin{aligned}
    \Vert f - P_N f_{TT} \Vert_{L^2_\mu({\bf I})} &\leq \Vert f - f_{TT} \Vert_{L^2_\mu({\bf I})} + \Vert f_{TT} - P_N f_{TT} \Vert_{L^2_\mu({\bf I})} \\
    &\leq \Vert f \Vert_{\mathcal{H}_\mu^k({\bf I})} \sqrt{ (d-1) \frac{(r+1)^{-(k-1)}}{k-1} } + C(k)N^{-k}\vert f_{TT} \vert_{{\bf I},\mu,k} \;.
  \end{aligned}
\end{equation}}
This result shows that convergence is driven by the selection of the rank $r$ and the polynomial degree $N$, and that it improves for functions with increasing regularity. Thus we can efficiently compute the expansion coefficients $\bm{\mathcal{C}}$ by \eqref{eq:CoreProjection} and obtain an approximation $P_{\mathbf{N}}f_{TT}$ that converges spectrally.

In practice, the projector $P_{\mathbf{N}}$ is replaced by the discrete projector $\widetilde{P}_{\mathbf{N}}$ \eqref{eq:DiscreteSpectralExpansion}, such that the coefficients $\{ \beta_n \}$ representing projections of the cores are approximated as
\begin{equation}\label{eq:CoreProjection}
  \beta_n(\alpha_{n-1},i_n,\alpha_{n}) \approx \hat{\beta}_n(\alpha_{n-1},i_n,\alpha_{n}) = \sum_{j=0}^{N_n} \gamma_n(\alpha_{n-1},x_n^{(j)},\alpha_n)\phi_{i_n}(x_n^{(j)}) w_n^{(j)} \, ,
\end{equation}
where $\{(x_n^{(j)},w_n^{(j)})\}_{j=0}^{N_n}$ are appropriate quadrature nodes and weights (e.g., Gauss rules, as described in Section~\ref{sec:approximation-theory}) for dimension $n$. This numerical approximation requires evaluating the cores of the FTT decomposition at the quadrature points. But these values $\gamma_n(\alpha_{n-1},x_n^{(j)},\alpha_n)$ in fact are approximated by the cores of the \textit{discrete} TT approximation of $f(\bm{\mathcal{X}})$ -- that is, $\bm{\mathcal{A}}^w_{TT}$, as described in Section~\ref{sec:dttftt}. The end result of this procedure can be viewed as the TT representation $\bm{\mathcal{C}}_{TT}$ of the spectral coefficient tensor $\bm{\mathcal{C}}$. The computational procedure is summarized in Procedure \ref{alg:FTT-projection-construction}.

\begin{algorithm}[t]
\caption{\texttt{FTT-projection-construction}}
\label{alg:FTT-projection-construction}
\begin{algorithmic}[1]
\Require Function $f:{\bf I} \rightarrow \mathbb{R}$; measure $\mu = \prod_{n=1}^d \mu_n$; integers ${\mathbf{N}}=\{ N_n \}_{n=1}^d $ denoting the polynomial degrees of approximation; univariate basis functions $\left\{ \left\{ \phi_{i_n,n} \right\}_{i_n=0}^{N_n} \right\}_{n=1}^d$ orthogonal with respect to $\mu_n$; \texttt{DMRG-cross} approximation tolerance $\varepsilon$.
\Ensure $\bm{\mathcal{C}}_{TT}(i_1,\ldots,i_d) = \sum_{\alpha_0,\ldots,\alpha_d=1}^{\bf r} \hat{\beta}_1(\alpha_0,i_1,\alpha_1) \cdots \hat{\beta}_d(\alpha_{d-1},i_d,\alpha_d)$, the TT-de\-com\-po\-si\-tion of the tensor of expansion coefficients.
\State Determine the univariate quadrature nodes and weights in each dimension, $ \left\{ ({\bf x}_n, {\bf w}_n) \right\}_{n=1}^d $,  where ${\bf x}_n = \{x_n^{(i)}\}_{i=0}^{N_n}$ and ${\bf w}_n = \{w_n^{(i)}\}_{i=0}^{N_n}$
\State Construct the $\varepsilon$--accurate approximation $\bm{\mathcal{B}}_{TT}$ of $h\left(\bm{\mathcal{X}}_{\bf i}\right) = f\left(\bm{\mathcal{X}}_{\bf i}\right)\sqrt{\bm{\mathcal{W}}_{\bf i}}$ using \texttt{TT-DMRG-cross}
\State Recover the approximation of $f(\bm{\mathcal{X}})$ as $\bm{\mathcal{A}}^w_{TT}=\bm{\mathcal{B}}_{TT}/\sqrt{\bm{\mathcal{W}}}$, with cores $\{ G_n \}_{n=1}^d$ and associated TT-ranks ${\bf r}$
\For{$n:=1$ to $d$}
  \For{$i_n:=0$ to $N_n$}
    \ForAll{$(\alpha_{n-1},\alpha_{n}) \in [0,r_{n-1}] \times [0,r_n] $}
      \State $\hat{\beta}_n(\alpha_{n-1},i_n,\alpha_{n}) = \sum_{j=0}^{N_n} G_n(\alpha_{n-1},j,\alpha_n)\phi_{i_n,n}(x_n^{(j)}) w_n^{(j)}$
    \EndFor
  \EndFor
\EndFor
\State \Return $\left\{ \hat{\beta}_n \right\}_{n=1}^d $
\end{algorithmic}
\end{algorithm}

Once Procedure~\ref{alg:FTT-projection-construction} (\texttt{FTT-projection-construction}) has been run, the spectral TT approximation can be evaluated at an arbitrary point $\mathbf{y}=\{y_1,\ldots,y_d\} \in {\bf I}$ by the procedure described in Procedure \ref{alg:FTT-projection-evaluation}.

\begin{algorithm}
\caption{\texttt{FTT-projection-evaluation}}
\label{alg:FTT-projection-evaluation}
\begin{algorithmic}[1]
\Require Cores $\left\{ \hat{\beta}_n(\alpha_{n-1},i_n,\alpha_n) \right\}_{n=1}^d $ obtained through \texttt{FTT\--projection\--con\-struc\-tion}; $N^y$ evaluation points $\mathbf{y}^{(i)} :=\{y_1^{(i)},\ldots,y_d^{(i)}\} \in {\bf I}$, $i \in [1, N^y]$, collected in the $N^y\times d$ matrix $\mathbf{Y} :=\{\mathbf{y}_1,\ldots,\mathbf{y}_d\}$.

\Ensure Polynomial approximation $\widetilde{P}_{\mathbf{N}} f_{TT}(\mathbf{Y})$ of $f(\mathbf{Y})$
\For{$n:=1$ to $d$}
  \For{$i:=1$ to $N^y$}
    \ForAll{$(\alpha_{n-1},\alpha_{n}) \in [0,r_{n-1}] \times [0,r_n] $}
      \State $\hat{G}_n(\alpha_{n-1}, i ,\alpha_{n}) = \sum_{j=0}^{N_n} \hat{\beta}_n(\alpha_{n-1},j,\alpha_n) \phi_{j,n}( y_n^{(i)}) $
    \EndFor
  \EndFor
\EndFor
\State $\bm{\mathcal{B}}_{TT}(i_1,\ldots,i_d) = \sum_{\alpha_0,\ldots,\alpha_d=1}^{\bf r} \hat{G}_1(\alpha_0,i_1,\alpha_1) \cdots \hat{G}_d(\alpha_{d-1},i_d,\alpha_d)$
\State \Return $\widetilde{P}_{\mathbf{N}} f_{TT}(\mathbf{Y}) := \left\{ \bm{\mathcal{B}}_{TT}(i,\ldots,i) \right\}_{i=1}^{N^y}$
\end{algorithmic}
\end{algorithm}

\subsubsection{Functional tensor-train interpolation}
Function interpolation can easily be extended to tensors, and the tensor-train format can be exploited to save computation and storage costs. We will first consider linear interpolation, using the notation of Section \ref{sec:interpolation}. Let $\bm{\mathcal{X}} = \times_{j=1}^d \mathbf{x}_j $ be a $N^x_1 \times \cdots \times N^x_d $ tensor of candidate interpolation nodes where the function $f$ can be evaluated, and let the matrix $\mathbf{Y}=\{ \mathbf{y}_1,\ldots,\mathbf{y}_d \}$ of size $N^y\times d$ represent a set of $N^y$ points where one wishes to evaluate the approximation of $f$. Define $\bm{\mathcal{Y}} = \times_{j=1}^d \mathbf{y}_j $.
%
%
An approximation of $f(\mathbf{Y})$ can be computed using the interpolation operator \eqref{eq:multilin-interpolation} from the grid $\bm{\mathcal{X}}$ to the grid $\bm{\mathcal{Y}}$
\begin{equation}\label{eq:ttinterp1}
  f(\bm{\mathcal{Y}}) \simeq \left( I_{\mathbf{N}}f \right)(\bm{\mathcal{Y}}) = \mathbf{E} f(\bm{\mathcal{X}}), \qquad \mathbf{E} = E^{(1)} \otimes \cdots \otimes E^{(d)},
\end{equation}
where $E^{(k)}$ is a $N^y \times N^x_{k}$ matrix defined by $E^{(k)}(i,j)=e^{(k)}_j({y}_k^{(i)}) $ as in \eqref{eq:multilin-shapefunc}, and then extracting only the diagonal of the tensor $f(\bm{\mathcal{Y}})$: $f(\mathbf{Y}) \simeq \left\{ \left( I_{\mathbf{N}}f \right)(\bm{\mathcal{Y}})_{i,\ldots,i} \right\}_{i=1}^{N^y}$. 
This leads to multi-linear interpolation on hypercubic elements. If we use the FTT approximation $f_{TT}$ instead of $f$ in \eqref{eq:ttinterp1}, we obtain
\begin{eqnarray} 
\label{eq:ttinterp_lin}
      \left( I_{\mathbf{N}}f_{TT} \right)(\bm{\mathcal{Y}})  = \mathbf{E} f_{TT}(\bm{\mathcal{X}}) & = & 
\mathbf{E} \left [ \sum_{\alpha=0,\ldots,\alpha_d = 1}^{\bf r} \gamma_1(\alpha_0,{\bf x}_1,\alpha_1) \cdots \gamma_d(\alpha_{d-1},{\bf x}_d,\alpha_d) \right ] \\
      &=& \sum_{\alpha=0,\ldots,\alpha_d = 1}^{\bf r} \beta_1(\alpha_0,\mathbf{y}_1,\alpha_1) \cdots \beta_d(\alpha_{d-1},\mathbf{y}_d,\alpha_d), \nonumber
 \end{eqnarray}
with
\begin{equation*}
  \beta_n(\alpha_{n-1},\mathbf{y}_n,\alpha_n) = E^{(n)} \gamma_n(\alpha_{n-1},{\bf x}_n,\alpha_n) \; .
\end{equation*}
Thus, instead of working with the tensor $\mathbf{E}$, we can work with the more manageable matrices $\{ E^{(n)} \}_{n=1}^d$. This approach is described in Procedure \ref{alg:FTT-interpolation-evaluation}. The ``construction'' phase of this approximation corresponds simply to applying the \texttt{TT-DMRG-cross} algorithm to $f(\bm{\mathcal{X}})$ to obtain $\bm{\mathcal{A}}^w_{TT}$, as described in Section~\ref{sec:dttftt}. The listing of \texttt{FTT-in\-ter\-po\-la\-tion-con\-struc\-tion} is thus omitted.
The basis functions \eqref{eq:multilin-shapefunc} yield quadratic convergence of the interpolant to the target function. Thus, for $k>d-1$ and $f \in \mathcal{H}_\mu^k({\bf I})$,
\rvnote*{\#1-1}{
\begin{equation}
    \Vert f - I_N f_{TT} \Vert_{L^2_\mu({\bf I})} \leq \Vert f \Vert_{\mathcal{H}_\mu^k({\bf I})} \sqrt{ (d-1) \frac{(r+1)^{-(k-1)}}{k-1} } + CN^{-2}\vert f_{TT} \vert_{{\bf I},\mu,2} \; .
\end{equation}}
Because these basis functions have local support (as opposed to the global support of the polynomials used for \texttt{FTT-projection}), errors due to singularities in $f$ do not pollute the entire domain.

\begin{algorithm}[t]
\caption{\texttt{FTT-interpolation-evaluation}}
\label{alg:FTT-interpolation-evaluation}
\begin{algorithmic}[1]
\Require Tensor of interpolation points $\bm{\mathcal{X}} = \times_{n=1}^d \mathbf{x}_n$, where $\mathbf{x}_n = \{x_n^{(i)} \}_{i=1}^{N^x_n} \subseteq I_n$; 
$\varepsilon$--accurate approximation $\bm{\mathcal{A}}^w_{TT}$ (in general) or $\bm{\mathcal{A}}_{TT}$ (uniform $\mu$, linear interpolation, equispaced points) of $f(\bm{\mathcal{X}})$ obtained by \texttt{TT-DMRG-cross}, with cores $\{ G_n \}_{n=1}^d$ and TT-ranks ${\bf r}$; 
evaluation points $\mathbf{y}^{(i)} :=\{y_1^{(i)},\ldots,y_d^{(i)}\} \in {\bf I}$, $i \in [1, N^y]$, collected in the $N^y\times d$ matrix $\mathbf{Y} :=\{\mathbf{y}_1,\ldots,\mathbf{y}_d\}$
\Ensure Interpolated approximation $I_{\mathbf{N}}f_{TT}(\mathbf{Y})$ or $\Pi_{\mathbf{N}}f_{TT}(\mathbf{Y})$ of $f(\mathbf{Y})$
\State Construct list $\left\{ L^{(i)} \right\}_{i=1}^d $ of $N^y \times N^x_{i}$ (linear or Lagrange) interpolation matrices from ${\bf x}_i$ to $\mathbf{y}_i$ 
\For{$n:=1$ to $d$}
  \ForAll{$(\alpha_{n-1},\alpha_{n}) \in [0,r_{n-1}] \times [0,r_n] $}
    \State $\hat{G}_n(\alpha_{n-1},:,\alpha_{n}) = L^{(n)} G_n(\alpha_{n-1},:,\alpha_n)$
  \EndFor
\EndFor
\State $\bm{\mathcal{B}}_{TT}(i_1,\ldots,i_d) = \sum_{\alpha_0,\ldots,\alpha_d=1}^{\bf r} \hat{G}_1(\alpha_0,i_1,\alpha_1) \cdots \hat{G}_d(\alpha_{d-1},i_d,\alpha_d)$
\State \Return $I_{\mathbf{N}}f_{TT}(\mathbf{Y}) := \left\{ \bm{\mathcal{B}}_{TT}(i,\ldots,i) \right\}_{i=1}^{N^y}$
\end{algorithmic}
\end{algorithm}

The same approach can be used for higher-order polynomial interpolation with Lagrange basis functions. 
The interpolated values can be obtained by extracting the diagonal $f(\mathbf{Y}) \simeq \left\{ \left( \Pi_{\mathbf{N}}f_{TT} \right)(\bm{\mathcal{Y}})_{i,\ldots,i} \right\}_{i=1}^{N^y}$ of
\begin{equation}
  f(\bm{\mathcal{Y}}) \simeq \left( \Pi_{\mathbf{N}}f \right)(\bm{\mathcal{Y}}) = {\bf L} f(\bm{\mathcal{X}}) , \qquad {\bf L} = L^{(1)} \otimes \cdots \otimes L^{(d)},
\end{equation}
where $L^{(k)}$ is the $N^y \times N^x_{k}$ Lagrange interpolation matrix \cite{Kopriva2009}. This interpolation is not carried out directly in high dimensions; as in the linear interpolation case, we only need to perform one-dimensional interpolations of the cores, i.e.,
\begin{equation}
  \label{eq:ttinterp_poly}
  \begin{aligned}
      \left( \Pi_{\mathbf{N}}f_{TT}\right)(\bm{\mathcal{Y}}) = {\bf L} f_{TT}(\bm{\mathcal{X}}) & = \sum_{\alpha_0,\ldots,\alpha_d = 1}^{\bf r} \beta_1(\alpha_0,\mathbf{y}_1,\alpha_1) \cdots \beta_d(\alpha_{d-1},\mathbf{y}_d,\alpha_d) \; , \\
\text{with} \     \beta_n(\alpha_{n-1},\mathbf{y}_n,\alpha_n) & = L^{(n)} \gamma_n(\alpha_{n-1},{\bf x}_n,\alpha_n) \; .
\end{aligned}
\end{equation}
Again, the evaluation procedure is detailed in Procedure~\ref{alg:FTT-interpolation-evaluation}. Convergence of the FTT interpolant is again dictated by the regularity of the function $f$. For $k > d-1$ and $f \in \mathcal{H}_\mu^k({\bf I})$, we have
\rvnote*{\#1-1}{
\begin{equation}
    \Vert f - \Pi_N f_{TT} \Vert_{L^2_\mu({\bf I})} \leq \Vert f \Vert_{\mathcal{H}_\mu^k({\bf I})} \sqrt{ (d-1) \frac{(r+1)^{-(k-1)}}{k-1} } + C(k)N^{-k}\vert f_{TT} \vert_{{\bf I},\mu,k} \; .
\end{equation}}

\subsubsection{Summary of algorithms}\label{sec:algorithm}

The preceding algorithms produce approximations of $f$ that involve both a (truncated) FTT approximation and polynomial (or piecewise linear) approximations of the FTT cores. We term these \textit{spectral tensor-train} (STT) approximations and summarize the algorithms as follows.

Suppose we have a function $f: {\bf I} \rightarrow \mathbb{R}$ where ${\bf I} = \times_{i=1}^d I_i$ and $I_i \subseteq \mathbb{R}$, for $i=1 \ldots d$. We would like to construct an STT approximation of $f$ and to evaluate this approximation on an independent set of points $\mathbf{Y}$. The \textit{construction} and \textit{evaluation} of the approximation involve the following steps:
\begin{enumerate}
\item Select a suitable set of candidate nodes $\bm{\mathcal{X}} = \times_{n=1}^d {\bf x}_n$ according to the type of approximation to be constructed; typically these are tensor-product quadrature or interpolation nodes.
\item In the projection approach, construct the approximation using Procedure \ref{alg:FTT-projection-construction}. In the interpolation approach, directly construct the approximation
$\bm{\mathcal{A}}^w_{TT}$ 
by applying \texttt{TT-DMRG-cross} to 
$h(\bm{\mathcal{X}})$, as described in Section~\ref{sec:dttftt}. 
In both approaches, we apply \texttt{TT-DMRG-cross} to the \textit{quantics folding} of the relevant tensors. This provides important performance improvements, particularly in low dimensions where \texttt{TT-DMRG-cross} would otherwise require taking the SVD of $f$ evaluated on hyperplanes.

\item Evaluate the the spectral tensor-train approximation on $\mathbf{Y}$ using Procedure \ref{alg:FTT-projection-evaluation} for the projection approach or using Procedure \ref{alg:FTT-interpolation-evaluation} for linear or Lagrange interpolation.
\end{enumerate}
Below, we will refer to the \texttt{FTT-projection} and the \texttt{FTT-interpolation} algorithms as the combination of the two corresponding steps of construction and evaluation. Our implementation of these algorithms uses data structures to cache computed values and to store partially computed decompositions. It also supports parallel evaluation of $f$ during the execution of \texttt{TT-DMRG-cross}, using the MPI protocol.


\section{Numerical examples}\label{sec:numerical-examples}

We now apply the spectral tensor-train approximation to several high dimensional functions. The construction of the approximation $\bm{\mathcal{A}}^w_{TT}  \simeq \bm{\mathcal{A}} = f(\bm{\mathcal{X}})$ is obtained through the application of the \texttt{TT-DMRG-cross} algorithm to the quantics folding of $\bm{\mathcal{B}} = h(\bm{\mathcal{X}})$, which leads to a sparser selection of the evaluation points. The quality of these approximations will be evaluated using the relative $L^2$ error:
\begin{equation}\label{eq:L2err}
  e_\text{rel} := {\Vert f - \mathcal{L}f_{TT} \Vert_{L^2_\mu({\bf I})}}/{\Vert f \Vert_{L^2_\mu({\bf I})}} ,
\end{equation}
where $\mathcal{L}$ is one of the projection ($P_N$) or interpolation ($I_N$, $\Pi_N$) operators. Integrals in the numerator and denominator of \eqref{eq:L2err} are estimated using Monte Carlo, with the number of samples chosen so that the relative error in $e_{\text{rel}}$ is less than $10^{-2}$.

\subsection{Genz functions and modified Genz functions}
The Genz functions \cite{genz1983a,Genz1987} are frequently used to evaluate function approximation schemes. They are defined on $[0,1]^d$, equipped with the uniform measure, as follows:
\begin{equation}
  \label{eq:GenzFunctions}
  \begin{aligned}
    f_1({\bf x}) &= \cos\left( 2\pi w_1 + \sum_{i=1}^d c_i x_i \right)\;, \quad &f_2({\bf x}) &= \prod_{i=1}^d \left( c_i^{-2} + (x_i + w_i)^2 \right)^{-1}\;,\\
    f_3({\bf x}) &= \left( 1 + \sum_{i=1}^d c_i x_i \right)^{-(d+1)}\;, \quad &f_4({\bf x}) &= \exp \left( - \sum_{i=1}^d c_i^2 (x_i-w_i)^2 \right)\;,\\
    f_5({\bf x}) &= \exp \left( - \sum_{i=1}^d c_i^2 \vert x_i-w_i \vert \right)\;, \quad &f_6({\bf x}) &=
    \begin{cases}
      0 & \text{if } x_1 > w_1 \;\text{or}\; x_2 > w_2\;, \\
      \exp \left( \sum_{i=1}^d c_i x_i \right) & \text{otherwise}\;,
    \end{cases}
  \end{aligned}
\end{equation}
and are respectively known as `oscillatory,' `product peak,' `corner peak,' `Gaussian,' `continuous,' and `discontinuous' functions.
\begin{table}
  \centering
  \begin{tabular}[center]{c|cccccc}
    & $f_1$ & $f_2$ & $f_3$ & $f_4$ & $f_5$ & $f_6$ \\ \hline
    $b_j$ & 284.6 & 725.0 & 185.0 & 70.3 & 2040.0 & 430.0 \\
    $h_j$ & 1.5 & 2.0 & 2.0 & 1.0 & 2.0 & 2.0
  \end{tabular}
  \caption{Normalization parameters for the Genz functions.}
  \label{tab:genz-norm}
\end{table}
\noindent The parameters ${\bf w}$ are drawn uniformly from $[0,1]$ and act as a shift for the function. 
In the classical definition of the Genz functions \cite{genz1983a,Genz1987}, the parameters ${\bf c}$ are drawn uniformly from $[0,1]$ and then normalized such that $  d^{h_j} \Vert {\bf c} \Vert_1 = b_j$, with $j$ indexing the six Genz functions. The ``difficulty'' of approximating the functions increases monotonically with $b_j$. The scaling constants $h_j$ are defined as suggested in \cite{genz1983a,Genz1987}, while $b_j$ are selected in order to obtain the same test functions used for $d=10$ in \cite{Barthelmann2000}. These values are listed in Table \ref{tab:genz-norm}.

By the definition of the Genz functions above, it is apparent that the approximation difficulty (as measured by the number of function evaluations required to achieve a certain error) does not increase substantially with dimension. This is also confirmed by numerical experiments. 
As an example, consider the `Gaussian' function $f_4$. It has the rank one representation
\begin{equation}
  \label{eq:GenzFunctions:Gaussian:RankOne}
  f_4({\bf x}) = \exp \left( - \sum_{i=1}^d c_i^2 (x_i-w_i)^2 \right) = \prod_{i=1}^d \exp \left( -c_i^2 (x_i-w_i)^2 \right) \;.
\end{equation}
Recall that ${\bf c}$ is normalized such that $\Vert {\bf c} \Vert_1 = \frac{b_j}{d^{h_j}}$. Then, for $d \rightarrow \infty$ and for the values of $h_j$ and $b_j$ listed in Table \ref{tab:genz-norm}, $c_i \rightarrow 0$ and $f_4(\mathbf{x}) \rightarrow 1$. Thus, with higher dimensions $d$ the function becomes nearly constant and hence easier to approximate.

We would instead like to test the performance of the STT approximation on a set of functions whose ``difficulty'' continues growing with dimension. To this end, we use the definition \eqref{eq:GenzFunctions} of the Genz functions but refrain from normalizing the coefficients ${\bf c} \sim \mathcal{U}([0,1]^d)$. This choice produces functions that do not degenerate to constants with increasing $d$, and thus can be used for meaningful tests in higher dimensions. We will refer to these functions as the ``modified Genz functions.''

For the sake of analyzing the following numerical experiments, it is important to note that most of the Genz functions---modified or not---are analytically low rank, meaning that they can be exactly written in FTT format with finite rank. As noted above, the `Gaussian' Genz function \eqref{eq:GenzFunctions:Gaussian:RankOne} has a FTT rank of one, independent of $d$. In the same way, the `product peak,'  `continuous,' and `discontinuous' functions are FTT rank-one functions, while the `oscillatory' function is a FTT-rank-two function. In contrast, the `corner peak' function cannot be  represented with finite FTT rank, leading to a dependence of its numerical FTT rank on the dimension $d$.

The numerical experiments below are performed by randomly sampling 30 independent sets of parameters $\bf w$ and $\bf c$ for each Genz function and evaluating the relative $L^2$ error \eqref{eq:L2err} for each approximation. We will show the relationship between this error and the number of function evaluations employed, for different values of the input dimension $d$ and different polynomial degrees. Both the error and the number of function evaluations will vary depending on the particular function at hand. In particular, the number of function evaluations is driven by the procedure for obtaining a discrete TT approximation on the desired tensor grid using the \texttt{TT-DMRG-cross} algorithm (see Section~\ref{sec:tt-dmrg}).
We use a conservative value of $\varepsilon=10^{-10}$ for the target relative accuracy \eqref{eq:TT-SVD-error-2} of the \texttt{TT-DMRG-cross} approximation.

\subsubsection{\texttt{FTT-projection} of the modified Genz functions}

\afterpage{
\begin{figure}
  \centering
  \begin{subfigure}[b]{0.48\textwidth}
    \includegraphics[width=\textwidth]{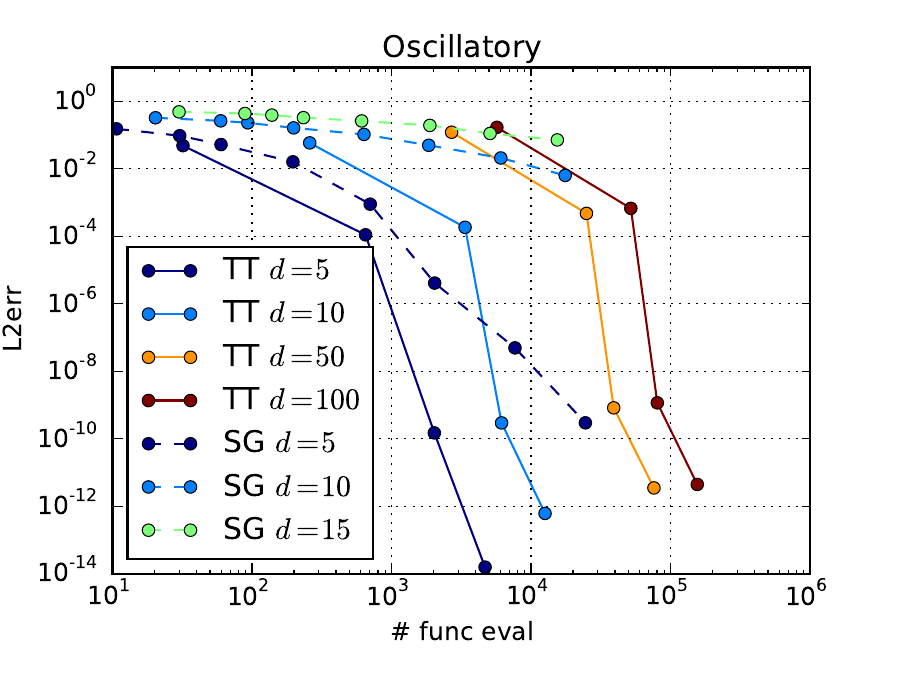}
    \label{fig:mod-genz-prj-oscillatory}
  \end{subfigure}
  ~
  \begin{subfigure}[b]{0.48\textwidth}
    \includegraphics[width=\textwidth]{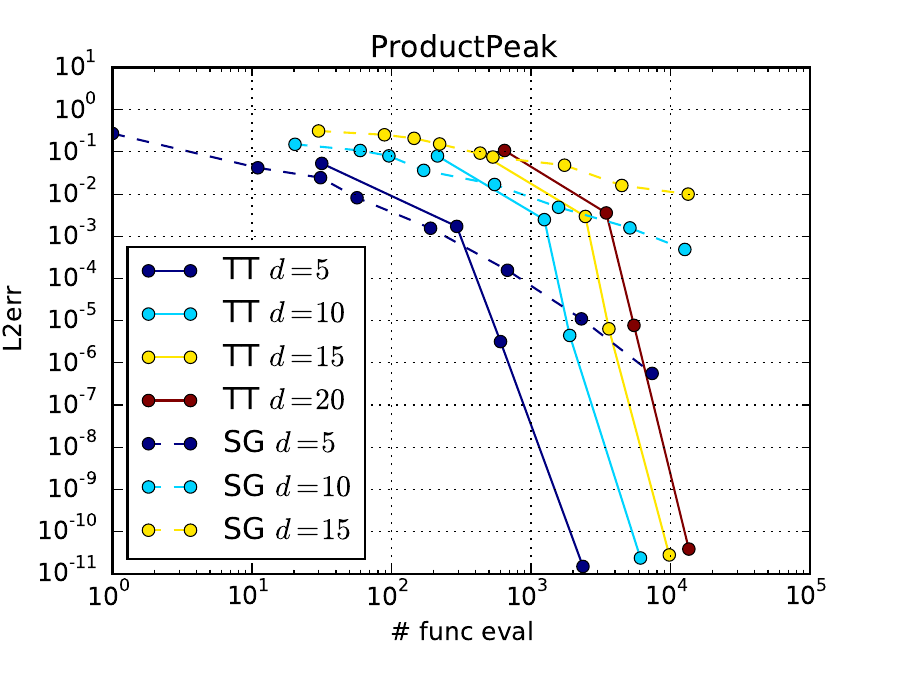}
    \label{fig:mod-genz-prj-productpeak}
  \end{subfigure}\\
  \begin{subfigure}[b]{0.48\textwidth}
    \includegraphics[width=\textwidth]{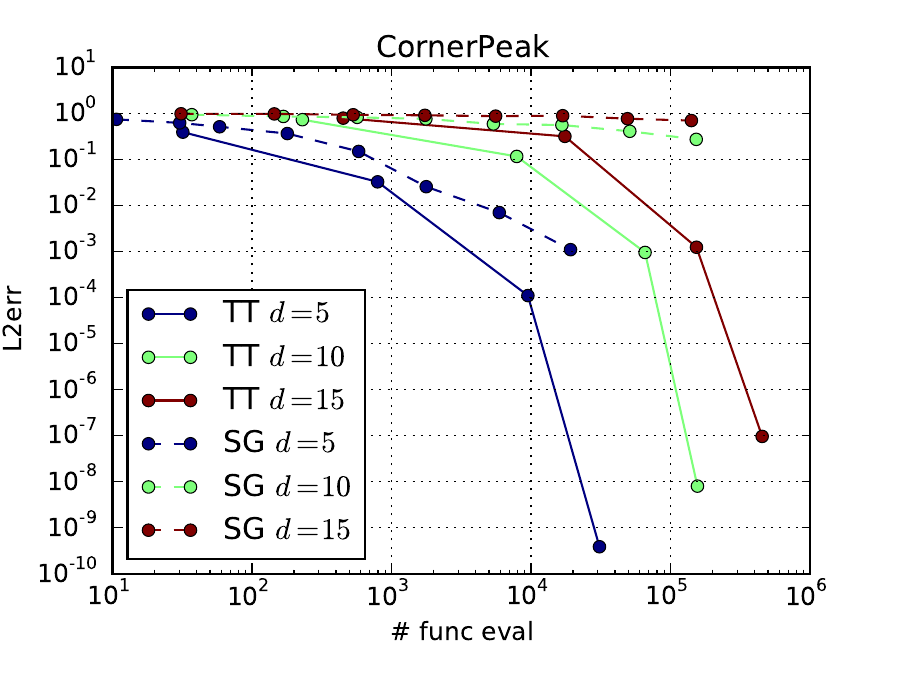}
    \label{fig:mod-genz-prj-cornerpeak}
  \end{subfigure}
  ~
  \begin{subfigure}[b]{0.48\textwidth}
    \includegraphics[width=\textwidth]{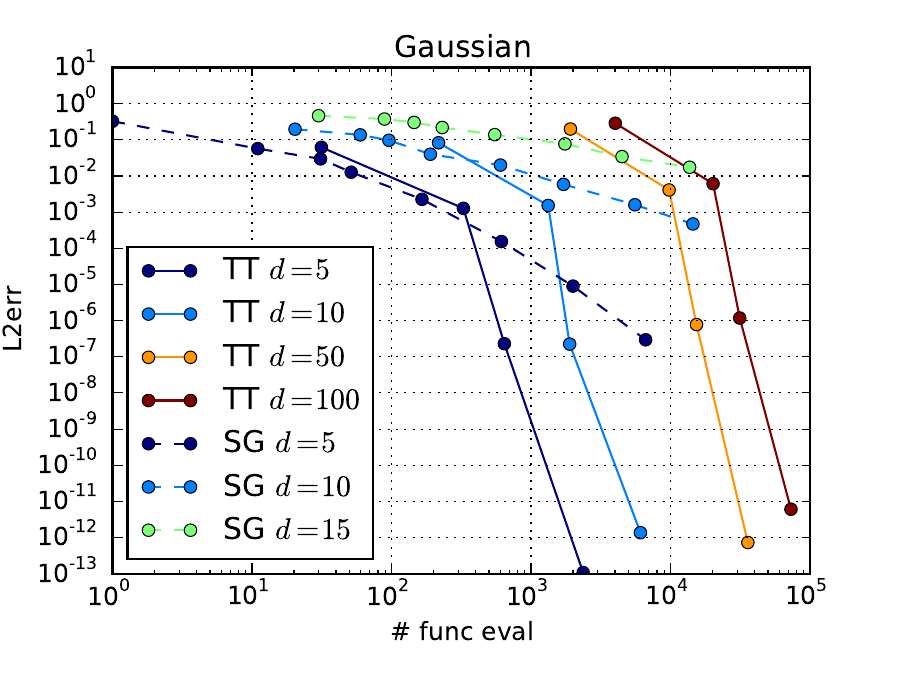}
    \label{fig:mod-genz-prj-gaussian}
  \end{subfigure}\\
  \begin{subfigure}[b]{0.48\textwidth}
    \includegraphics[width=\textwidth]{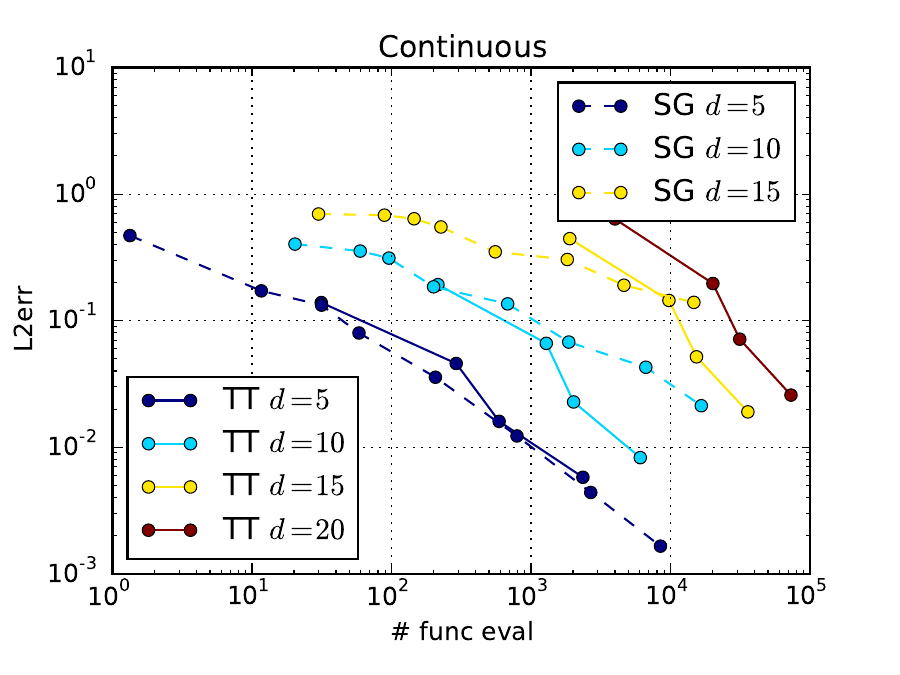}
    \label{fig:mod-genz-prj-continuous}
  \end{subfigure}
  ~
  \begin{subfigure}[b]{0.48\textwidth}
    \includegraphics[width=\textwidth]{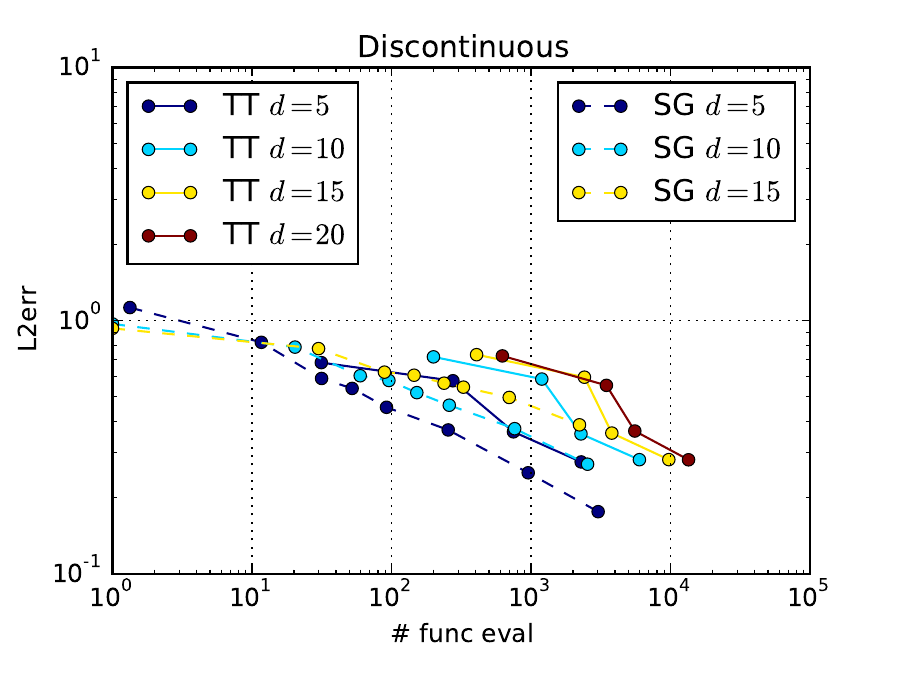}
    \label{fig:mod-genz-prj-discontinuous}
  \end{subfigure}
  \caption{\texttt{FTT-projection} approximation of the modified Genz functions. For exponentially increasing polynomial degree ($2^i-1$ for $i=1 \ldots 4$) and for varying dimensions $d$, we construct 30 realizations of each modified Genz function and evaluate the relative $L^2$ errors of their approximations. The circled dots represent the mean relative $L^2$ error and mean number of function evaluations for each polynomial degree. The figures compare the convergence rates of the \texttt{FTT-projection} and the anisotropic adaptive Smolyak algorithm \cite{Conrad2012}}
  \label{fig:mod-genz-proj}
\end{figure}
\clearpage
}

\afterpage{
\begin{figure}
  \centering
  \begin{subfigure}[b]{0.48\textwidth}
    \includegraphics[width=\textwidth]{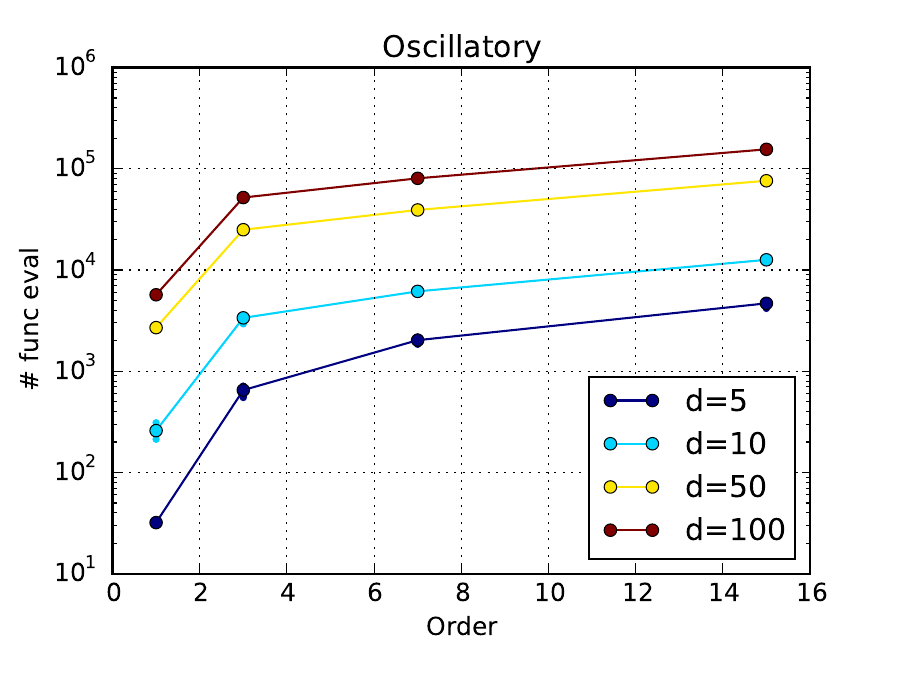}
    \label{fig:modGenz-prj-oscillatory-OrdVsFeval}
  \end{subfigure}
  ~
  \begin{subfigure}[b]{0.48\textwidth}
    \includegraphics[width=\textwidth]{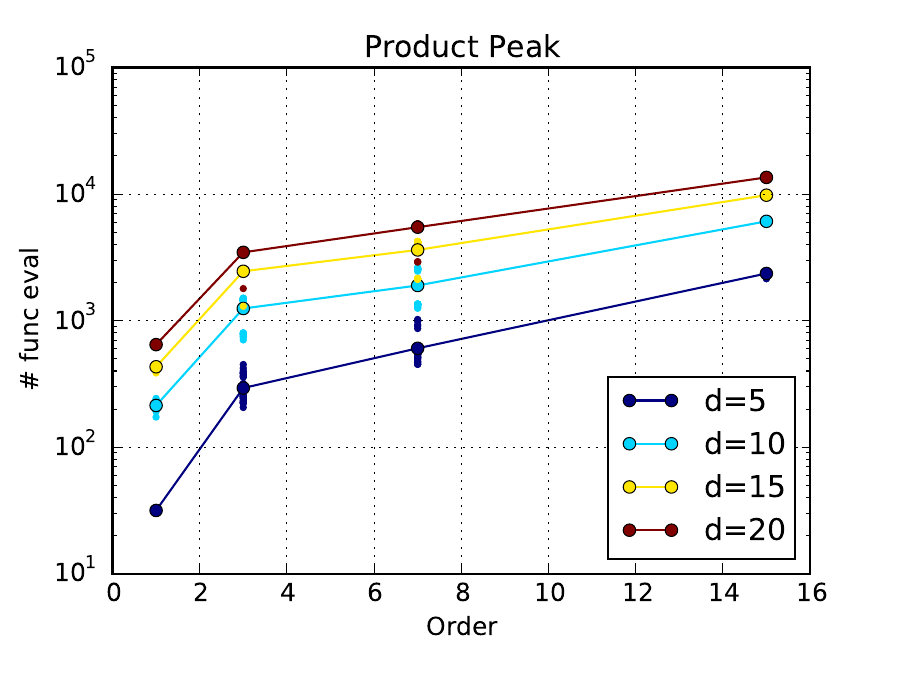}
    \label{fig:modGenz-prj-productpeak-OrdVsFeval}
  \end{subfigure}\\
  \begin{subfigure}[b]{0.48\textwidth}
    \includegraphics[width=\textwidth]{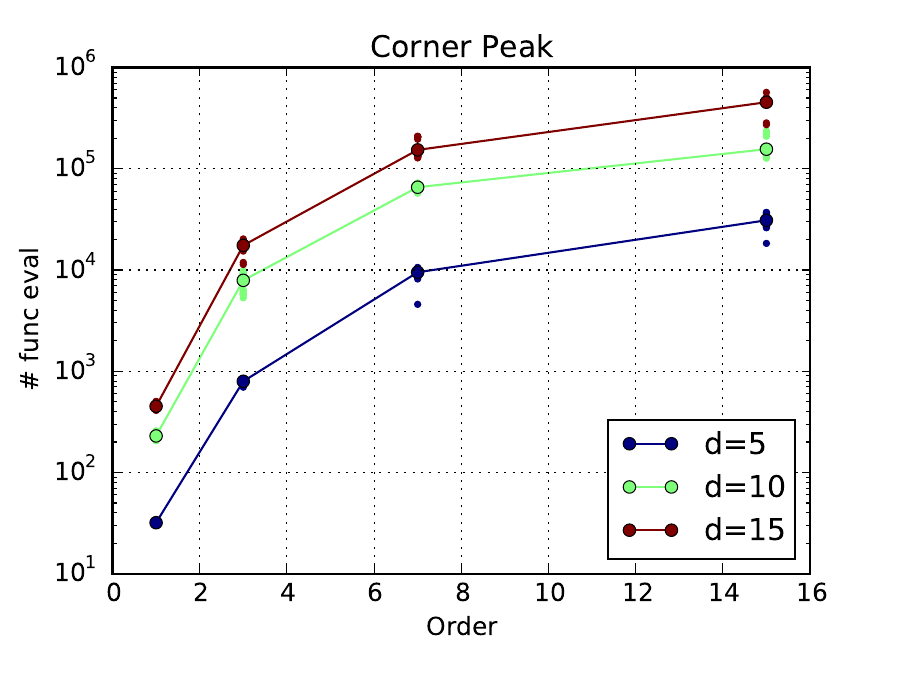}
    \label{fig:modGenz-prj-cornerpeak-OrdVsFeval}
  \end{subfigure}
  ~
  \begin{subfigure}[b]{0.48\textwidth}
    \includegraphics[width=\textwidth]{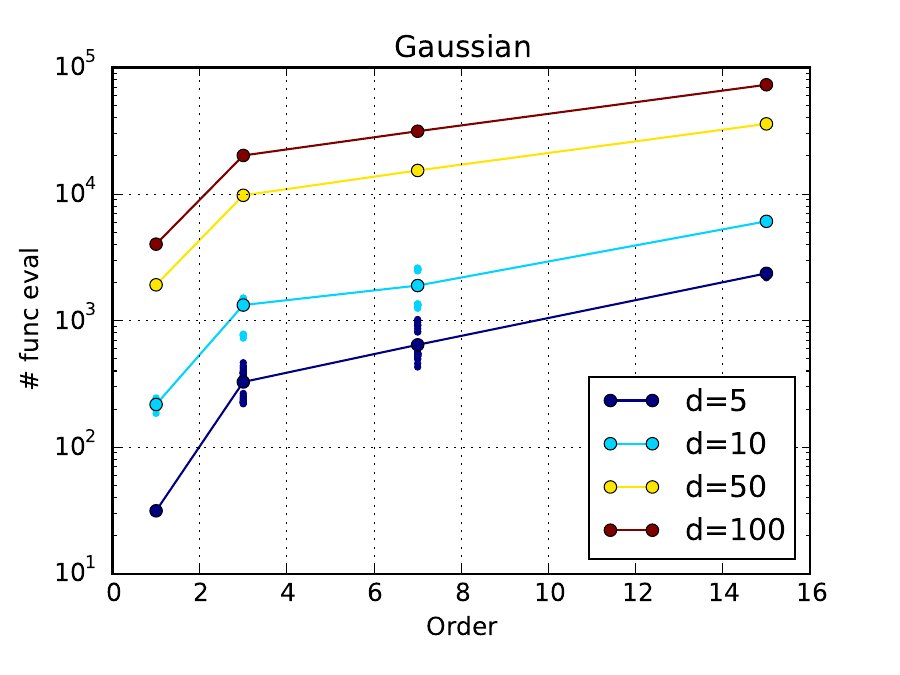}
    \label{fig:modGenz-prj-gaussian-OrdVsFeval}
  \end{subfigure}\\
  \begin{subfigure}[b]{0.48\textwidth}
    \includegraphics[width=\textwidth]{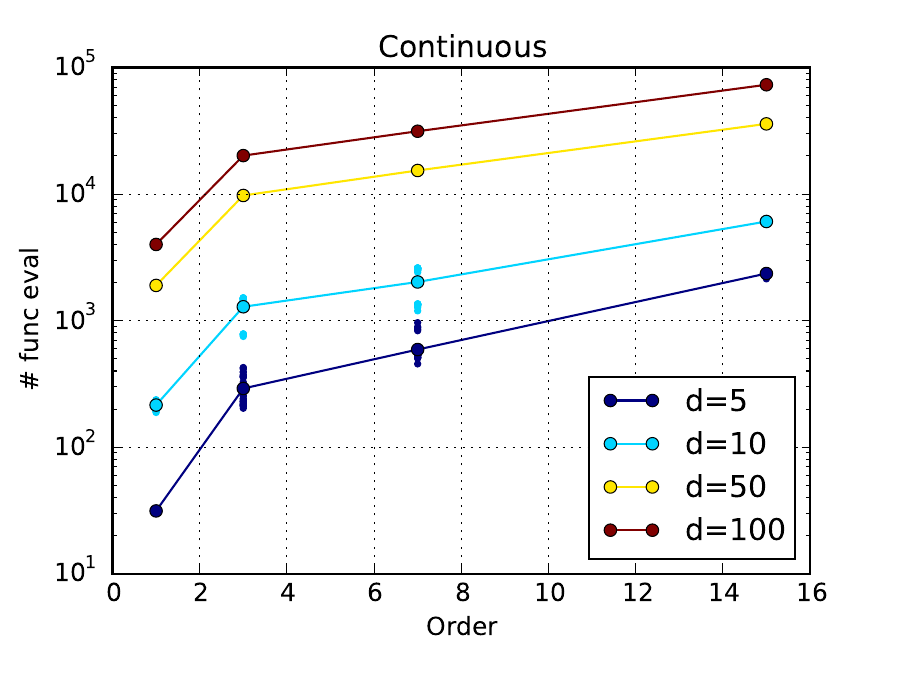}
    \label{fig:modGenz-prj-continuous-OrdVsFeval}
  \end{subfigure}
  ~
  \begin{subfigure}[b]{0.48\textwidth}
    \includegraphics[width=\textwidth]{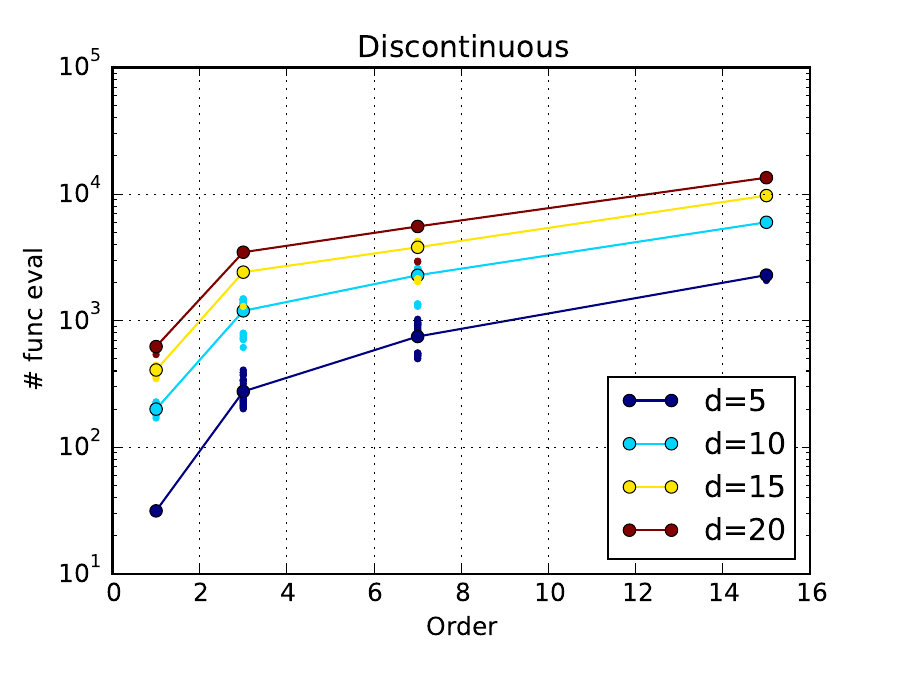}
    \label{fig:modGenz-prj-discontinuous-OrdVsFeval}
  \end{subfigure}
  \caption{\texttt{FTT-projection} approximation of the modified Genz functions. For exponentially increasing polynomial degree ($2^i-1$ for $i=1 \ldots 4$) and for varying dimensions $d$, we construct 30 realizations of each Genz function. The dots show the number of function evaluations required to construct an STT approximation of the specified polynomial degree.}
  \label{fig:modGenz-proj-OrdVsFeval}
\end{figure}
\clearpage
}

Our numerical tests consider dimensions $d$ ranging from $5$ to $100$ for functions $f_1$, $f_4$, and $f_5$. The `corner peak' function $f_3$ was tested only up to $d=15$ due to the higher computational effort required to build its approximations, as discussed below. For the `product peak' function $f_2$, we could not run tests for $d>20$ due to limited machine precision, because $f_2 \rightarrow 0$ as $d$ increases. The results are compared to approximations obtained using an anisotropic adaptive sparse grid algorithm \cite{Conrad2012}, with Gauss-Patterson quadrature rules \cite{Patterson1968}.

Figure \ref{fig:mod-genz-proj} shows the convergence of the \texttt{FTT-projection} approximation of the six modified Genz functions, for exponentially increasing polynomial degree ($N = 2^i-1$ for $i=1 \ldots 4$), isotropic across dimensions ($N_n =N$, $n=1\ldots d$), with Gauss-Legendre points/weights used for quadrature. In particular, we show the relative error of the approximations versus the number of function evaluations, for increasing polynomial degree. Figure \ref{fig:modGenz-proj-OrdVsFeval} shows the relationship between number of function evaluations and the degree of the polynomial basis, for varying dimension. The scatter of the points in the figures reflects the randomness in the coefficients of the modified Genz functions, the resulting polynomial approximation error, and the approximate fulfillment of the relative error criterion \eqref{eq:TT-SVD-error-2} by \texttt{TT-DMRG-cross}. Due to the interchangeability of the dimensions in the modified Genz functions, realizations of the error are more scattered for the lower-dimensional functions, as these functions are defined by fewer random parameters. As expected we observe a spectral convergence rate for the smooth functions $f_1$ through $f_4$. For the `continuous' modified Genz function, the convergence is only quadratic, since the function has a first-order discontinuity. Approximation of the `discontinuous' function shows very slow convergence, due to the use of a global polynomial basis for a function that is not even $\mathcal{C}^0$. 

The number of function evaluations required to achieve a given accuracy increases linearly with $d$ for functions with finite FTT-ranks that are independent of dimension (e.g., all the modified Genz functions except the `corner peak'). The absence of an exact finite-rank FTT decomposition for the `corner peak' function leads to a truncation effectively controlled by the quadrature level and the \texttt{DMRG} tolerance $\varepsilon$. This, in turn, leads to FTT approximation ranks that grow with dimension and thus a superlinear growth (in $d$) of the number of function evaluations.

The comparison to the sparse grid algorithm \cite{Conrad2012} shows dramatic improvements in performance. The convergence rate of the sparse grid algorithm analyzed is acceptable for functions of moderate dimension ($d=5$), but deteriorates considerably with increasing $d$. The convergence rate of \texttt{FTT-projection} is instead consistently better, even on the `corner peak' function where the numerical rank depends on dimension. It is important to stress that the functions analyzed here are mildly anisotropic and that the sparse grid method could perform better on more anisotropic functions. Nevertheless, very anisotropic functions are in practice effectively lower-dimensional, whereas the functions analyzed in this example are truly high-dimensional. Another important aspect of this comparison is that the computational complexity of the anisotropic \textit{adaptivity} of the sparse grid algorithm---not in terms of function evaluations, but rather algorithmic overhead---grows exponentially with dimension, because the set of active indices is defined over a high-dimensional surface. In contrast, the complexity of the \texttt{FTT-projection} algorithm grows only polynomially in terms of the dimension and the rank.

\subsubsection{\texttt{FTT-interpolation} of the modified Genz functions}

We have tested linear \texttt{FTT}-\texttt{interpolation} on all the modified Genz functions, with an exponentially increasing number of uniformly distributed points in each dimension, ranging from $2^1$ to $2^7$. For brevity, Figure \ref{fig:genz-lininterp} shows convergence results only for the `continuous' and `discontinuous' Genz functions. For the first four smooth Genz functions we observe at least second order convergence rates, as expected from the choice of a linear basis. The convergence of the  \texttt{FTT-interpolation} approximation to the `continuous' function is also second order, while the convergence rate for the `discontinuous' function is almost first order. Improved convergence for the latter, compared to Figure \ref{fig:mod-genz-proj}, is due to the local support of the selected basis functions, which prevents the discontinuity from globally corrupting the approximation.

\begin{figure}
  \centering
  \begin{subfigure}[b]{0.48\textwidth}
    \includegraphics[width=\textwidth]{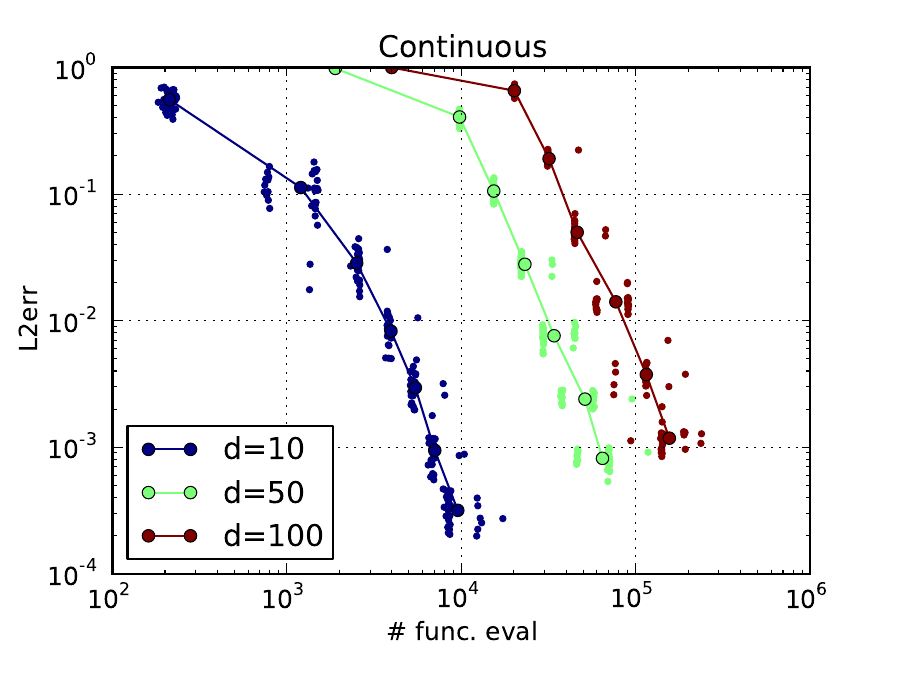}
    \label{fig:genz-lininterp-continuous}
  \end{subfigure}
  ~
  \begin{subfigure}[b]{0.48\textwidth}
    \includegraphics[width=\textwidth]{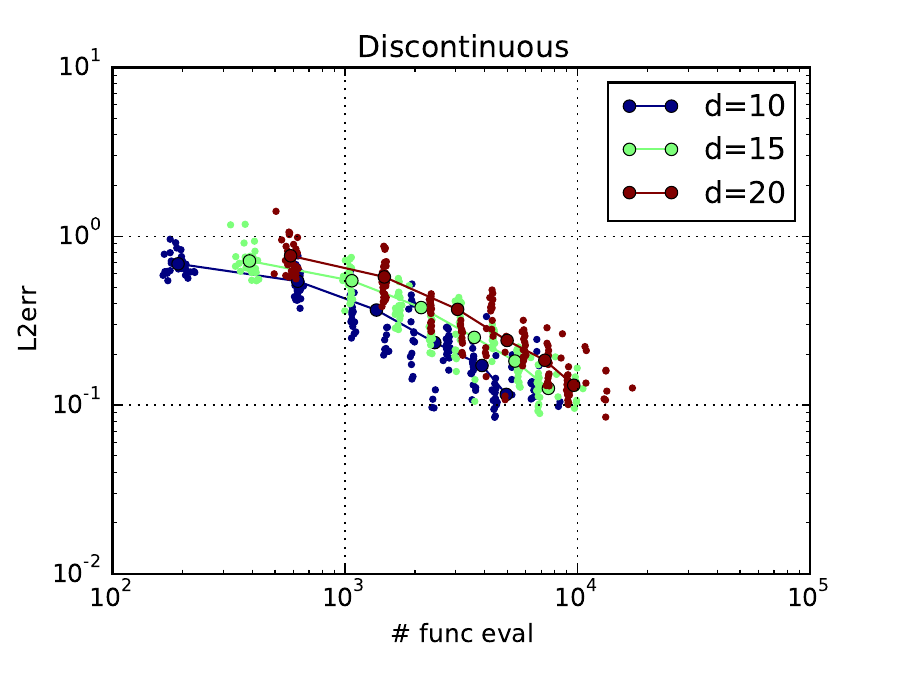}
    \label{fig:genz-lininterp-discontinuous}
  \end{subfigure}
  \caption{FTT linear interpolation of the `continuous' and `discontinuous' modified Genz functions. For exponentially numbers of uniformly distributed interpolation points ($2^1$ to $2^7$) and for varying dimensions $d$, we construct 30 realizations of each modified Genz function and evaluate the relative $L^2$ errors of their approximations. The scattered dots show the relative $L^2$ error versus the number of required function evaluations for each realization. The circled dots represent the mean relative $L^2$ error and mean number of function evaluations for each level of grid refinement.}
  \label{fig:genz-lininterp}
\end{figure}

We have also tested Lagrange \texttt{FTT-interpolation} for all the modified Genz functions; we omit the results here because they closely follow the results obtained with \texttt{FTT-projection}, already shown in Figure \ref{fig:mod-genz-proj}.

\subsection{\texttt{FTT-projection} and mixed Fourier modes}\label{sec:tens-train-proj-Fourier}

An important contrast between the STT approximation and sparse grid approximations is their behavior for mixed Fourier modes. It is well understood that sparse grid approximations are most effective for functions that are loosely coupled, i.e., that do not contain significant multiplicative terms involving several inputs at high polynomial degree. More precisely, the convergence of a sparse grid approximation deteriorates when the decay of the Fourier coefficients is slow for mixed modes.


\begin{figure}
  \centering
  \begin{subfigure}[b]{0.42\textwidth}
    \includegraphics[width=\textwidth]{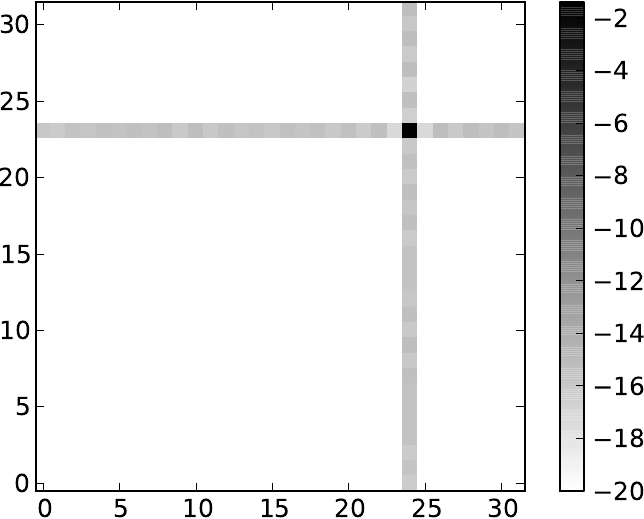}
    \caption{$f_1$: $d=2$, rank=1, fevals = $209/32^2$\\
      \indent $d=3$, rank=1, fevals = $626/32^3$\\
      \indent $d=4$, rank=1, fevals = $1210/32^4$\\
      \indent $d=5$, rank=1, fevals = $1442/32^5$}
    \label{fig:Mixed-Fourier-func-2d-f1}
  \end{subfigure}
   ~
  \begin{subfigure}[b]{0.42\textwidth}
    \includegraphics[width=\textwidth]{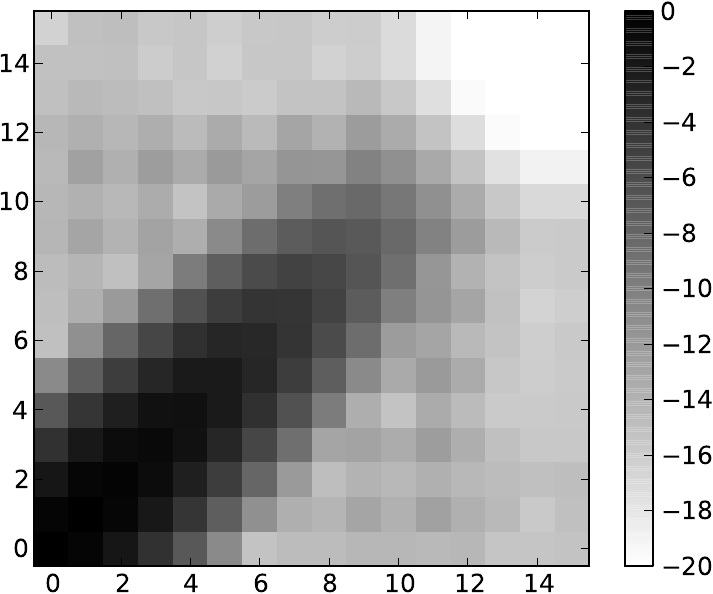}
    \caption{$f_2$: $d=2$, $J=[0,1]$, fevals = $256/16^2$\\
      \indent $d=5$, $J=[1,2]$, fevals = $3935/16^5$\\
      \indent $d=5$, $J=[0,4]$, fevals = $73307/16^5$\\
    ~}
    \label{fig:Mixed-Fourier-func-2d-f2}
  \end{subfigure}
  \caption{Magnitude of the Fourier coefficients, in $\log_{10}$ scale, for functions \eqref{eq:decay-functions}, obtained using the \texttt{TT-projection} algorithm with a tolerance of $\varepsilon=10^{-10}$. The corresponding maximum TT-rank and number of function evaluations/total grid size are listed for several dimensions $d$.}
  \label{fig:Mixed-Fourier-func-2d}
\end{figure}

We construct two ad hoc functions to highlight some properties of the \texttt{FTT-projection} when approximating functions with different decays in their Fourier coefficients. Consider functions defined on ${\bf I} = I_1 \times \cdots \times I_d$ where $I_i = [-1,1]$. 
Now consider the subset of indices $J=\{j_i\}_{i=1}^c \subseteq [1,\ldots,d]$. For every element of $J$, let $\{n_{j_i}\}_{i=1}^c>0$ be the maximum polynomial degree of the function in the $j_i$ direction. 
The functions are then defined as follows:
\begin{equation}\label{eq:decay-functions}
  \begin{aligned}
    f_1({\bf x}) &= \prod_{k=1}^c \phi_{l_k}(x_{j_k}) \;, \\
    f_2({\bf x}) &= \sum_{i_{j_1}=0}^{n_{j_1}} \cdots \sum_{i_{j_c}=0}^{n_{j_c}} \left[ \exp \left( - {\bf i}^\top {\bf \Sigma} {\bf i} \right) \prod_{k=1}^c \phi_{i_{j_k}}(x_{j_k}) \right] \;,
  \end{aligned}
\end{equation}
where ${\bf \Sigma}$ is a $c \times c$ matrix defining interactions between different dimensions, $\phi_{i}$ is the normalized univariate Legendre polynomial of degree $i$, and ${\bf i} = (i_{j_1},\ldots,i_{j_c})^\top$. To simplify the notation, we will set $n_{j_k} = n$ for all $j_k$.

The function $f_1$ has a single high-degree mixed Fourier mode as shown in Figure \ref{fig:Mixed-Fourier-func-2d-f1}; we use $d=c=2$, with $l_1 = 24$ and $l_2 =23$.  Despite this high degree, the rank of the function is correctly estimated to be one and thus very few sampling points are needed in order to achieve the required precision. The success of the STT approximation in this example highlights the fact that, unlike sparse grids, the spectral tensor-train always uses a fully tensorized set of basis functions.

The function $f_2$ is intended to have a slow decay of its mixed Fourier coefficients in the $J$ dimensions, but is constant along the remaining dimensions. For $d=2$ and $J=[0,1]$ we set 
$${\bf \Sigma}=
\left[\begin{array}{cc}
  1 & -0.9 \\
  -0.9 & 1
\end{array}\right].$$ The decay of the coefficients, as estimated using the \texttt{FTT-projection}, is shown in Figure \ref{fig:Mixed-Fourier-func-2d-f2}. The function has an high TT-rank, which leads to a complete sampling of the discrete tensor. We can also use this function  to experiment with the \textit{ordering} problem of the TT-decomposition. We let $d=5$ and use different combinations of indices in $J$. If $J$ contains two neighboring dimensions, $J=[1,2]$ in the example above, the TT ranks of the decomposition, obtained through numerical truncation, will be ${\bf r}=[1,1,11,1,1,1]$, where the maximum is attained between the cores $G_1$ and $G_2$. If instead we consider a $J$ containing non-neighboring dimensions, e.g., $J=[0,4]$ in Figure \ref{fig:Mixed-Fourier-func-2d-f2}, we obtain the same function but with reordered dimensions. Now the TT-ranks become ${\bf r}=[1,11,11,11,11,1]$. This happens due to the sequential construction of the TT-decomposition, where information can be propagated only from one core to the next. The example shows that the consequence of a poor ordering choice is an increased number of function evaluations, which grows with $r^2$. Importantly, however, this choice does not affect the accuracy of the approximation.

\subsection{Resolution of local features}
\begin{figure}
  \centering
  \begin{subfigure}[b]{0.48\textwidth}
    \includegraphics[width=\textwidth]{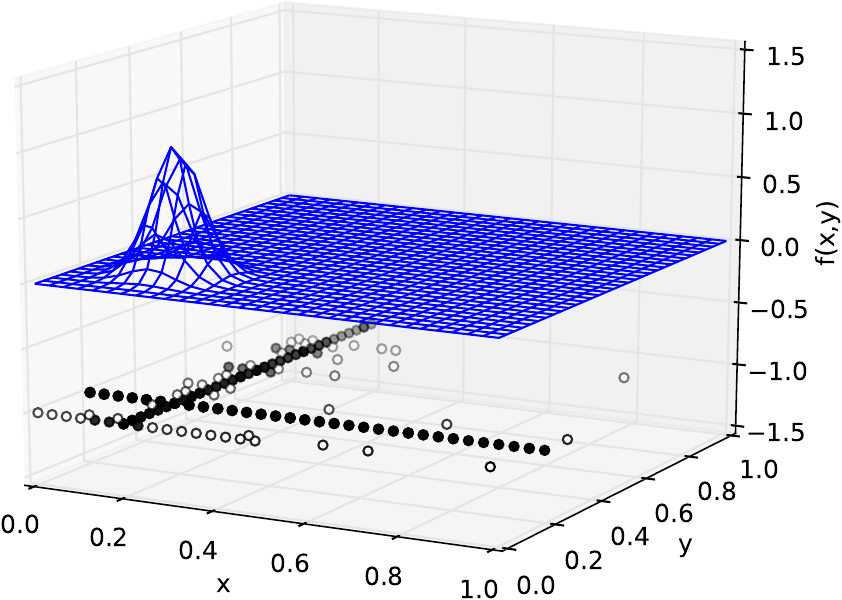}
    \caption{$d=2$, fevals = $119/32^2$}
    \label{fig:Features-2d}
  \end{subfigure}
   ~
  \begin{subfigure}[b]{0.48\textwidth}
    \includegraphics[width=\textwidth]{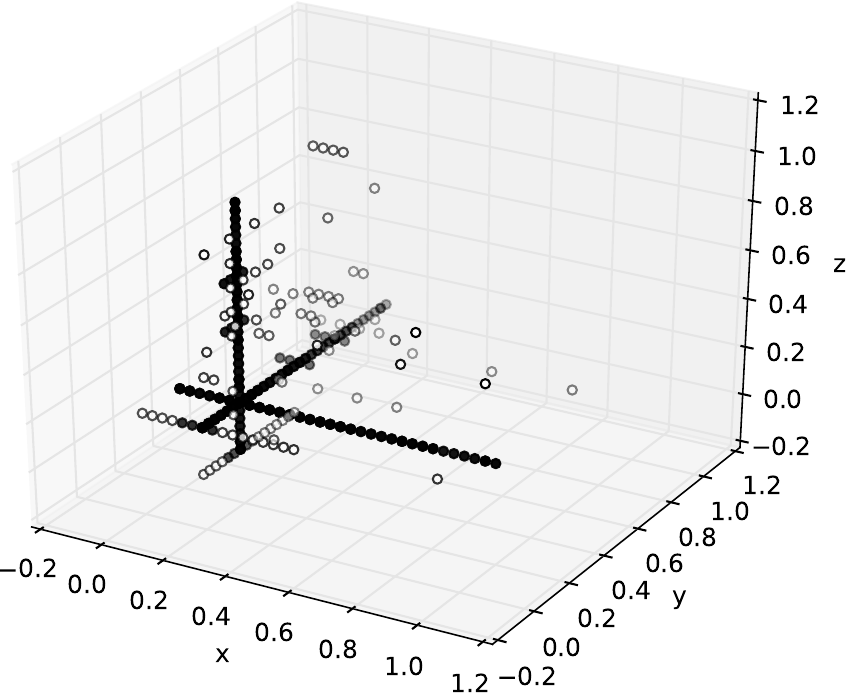}
    \caption{$d=3$, fevals = $214/32^3$}
    \label{fig:Features-3d}
  \end{subfigure}
  \caption{\texttt{TT-DMRG-cross} approximation of function \eqref{eq:local-feature}, which has a localized feature as shown in blue for $d=2$ on the left. The open and filled circles show the candidate points where the function has been evaluated. The filled circles are the points used in the final \texttt{TT-DMRG-cross} approximation.  \texttt{TT-DMRG-cross} detects the feature and clusters nodes around it in order to achieve a relative accuracy of $\varepsilon=10^{-10}$. The right figure shows the same test for $d=3$.}
  \label{fig:Features}
\end{figure}

Many functions of interest present local features that need to be resolved accurately. An \textit{a priori} clustering of nodes around a localized feature typically is not possible, because the location and shape of such a feature is unknown. The \texttt{TT-DMRG-cross} algorithm is able to overcome this problem because it adaptively selects the nodes that are relevant for the approximation, thus exploring the space with increasing knowledge about the structure of the function. As an illustrative example, consider the Gaussian bump
\begin{equation}\label{eq:local-feature}
  f({\bf x}) = \exp \left( - \frac{\vert {\bf x} - {\bf x}_0 \vert^2}{2l^2} \right) \quad .
\end{equation}
Let $d=2$, ${\bf x}_0 = (0.2,0.2)$, and $l=0.05$; the peak is thus off-center as shown in Figure \ref{fig:Features-2d}. We let $\bm{\mathcal{X}}$ be a uniform grid with 32 points per dimension and apply \texttt{TT-DMRG-cross} (with accuracy $\varepsilon=10^{-10}$) to the quantics folding of $f(\bm{\mathcal{X}})$. Open and filled circles show all the points at which the function is evaluated during iterations of \texttt{TT-DMRG-cross}. The filled circles correspond to the points selected in the last iteration.
Figure \ref{fig:Features-3d} shows the set of points used for $d=3$ and ${\bf x}_0 = (0.2,0.2,0.2)$. The same kind of clustering around the Gaussian bump is observed.

\subsection{Elliptic equation with random input data}\label{sec:elliptic-equation} 
In our final example, we approximate the solution of a linear elliptic PDE with a stochastic parameterized coefficient. Consider the Poisson equation on the unit square $\Gamma = [0,1]^2 \ni \mathbf{x}$,
\begin{equation}
  \label{eq:Poisson}
  \begin{cases}
    -\nabla \cdot \left(\kappa({\bf x},\omega) \nabla u({\bf x},\omega) \right) = f({\bf x}) & \text{in} \quad \Gamma \times \Omega \\
    u({\bf x},\omega) = 0 & \text{on} \quad \partial \Gamma \times \Omega 
  \end{cases} \; ,
\end{equation}
where $f({\bf x})=1$ is a deterministic source term and $\kappa$ is a log-normal random field defined on the probability space $(\Omega,\Sigma,\mu)$ by
\begin{equation}
  \label{eq:kappa}
  \kappa({\bf x},\omega) = \exp \left( g( {\bf x}, \omega) \right ) \;, \qquad g({\bf x},\omega) \sim \mathcal{N}\left( {\bf 0}, C_g( {\bf x}, {\bf x'} ) \right) \; .
\end{equation}
We characterize the normal random field $g \in L^2_\mu(\Omega;L^\infty(\Gamma))$ by the squared exponential covariance kernel:
\begin{equation}
  \label{eq:Covariance}
  C_g( {\bf x}, {\bf x'} ) = \int_\Omega g({\bf x},\omega) g({\bf x'},\omega) d\mu(\omega) = \sigma^2 \exp \left( - \frac{\Vert {\bf x} - {\bf x'} \Vert^2}{2l^2} \right) \;,
\end{equation}
where $l>0$ is the spatial correlation length of the field and $\sigma^2$ is a variance parameter. We decompose the random field through the Karhunen-Lo\`{e}ve (KL) expansion \cite{1978Loeve}
\begin{equation}
  \label{eq:KL-expansion}
  g({\bf x}, \omega) = \sum_{i=1}^\infty \sqrt{\lambda_i} \chi_i({\bf x}) Y_i(\omega) \;,
\end{equation}
where $Y_i\sim \mathcal{N}(0,1)$ and $\left\lbrace \lambda_i,\chi_i({\bf x}) \right\rbrace_{i=1}^\infty$ are the eigenvalues and eigenfunctions of the eigenvalue problem $\int_{\Gamma} C_g({\bf x},{\bf x'}) \chi_i({\bf x'}) d{\bf x'} = \lambda_i \chi_i({\bf x})$. The KL expansion is truncated in order to retain $95\%$ of the total variance, i.e., we find $d\in \mathbb{N}^+$ such that $\sum_{i=1}^d \lambda_i \geq 0.95 \sigma^2$. With a correlation length of $l=0.25$ and $\sigma^2 = 0.1$, this threshold requires $d=12$ terms in the KL expansion.
The use of the KL expansion allows \eqref{eq:Poisson} to be turned into a parametric problem, where we seek a solution $u \in L^2(\Gamma) \times L^2_{\mu}(\mathbb{R}^d)$. For the purpose of the current exercise, we will approximate this solution at a particular spatial location, i.e., seek approximations of $u({\bf x}_0,{\bf y})$ with ${\bf x}_0 = (0.75,0.25)$. 
%

\begin{figure}
  \centering
  \includegraphics[width=0.48\textwidth]{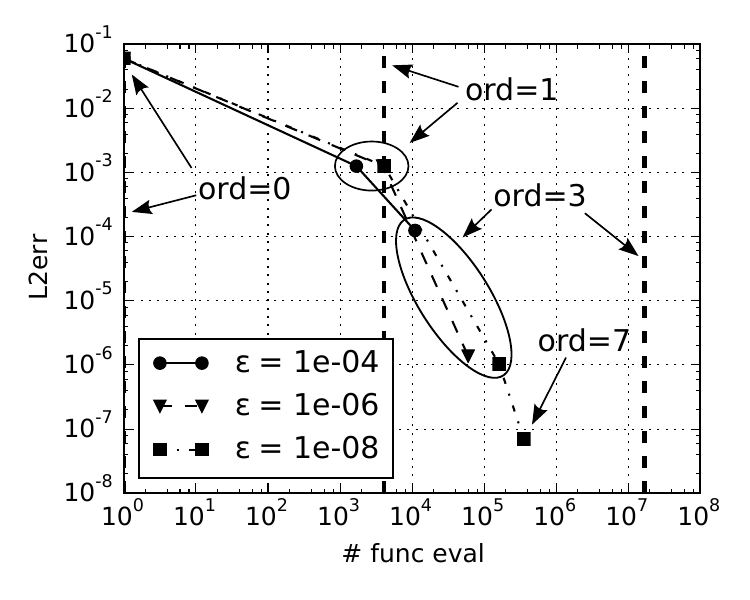}
  \caption{Convergence of the \texttt{FTT-projection} of orders 0, 1, 3, and 7 for different target accuracies selected. The vertical dashed lines show the number of function evaluations that would be required to attain a full tensor approximation.} 
  \label{fig:Poisson-convergence}
\end{figure}

We construct a surrogate using \texttt{FTT-projection} with Hermite polynomial basis functions, where $\bm{\mathcal{X}}$ is a full tensor of Gauss-Hermite quadrature points. We consider polynomial degrees of 0, 1, 3, and 7, and the corresponding tensors of size $1^d$, $2^d$, $4^d$, and $8^d$. Figure \ref{fig:Poisson-convergence} shows the convergence of our approximation in terms of the relative $L^2$ error \eqref{eq:L2err}, for different polynomial degrees and for different settings of the \texttt{TT-DMRG-cross} approximation tolerance $\varepsilon$. We see that the $L^2$ accuracy of the function approximation improves spectrally until reaching a plateau that matches $\varepsilon$ closely; beyond this plateau, an increase in the polynomial degree of the surrogate provides no further improvement, i.e., the convergence plot flattens at a relative accuracy that is $\mathcal{O}(\varepsilon)$. It is also interesting that, for a given polynomial degree and a desired relative $L^2$ error, the most efficient way of achieving this error is to choose $\varepsilon$ of the same order as this error. In other words, just as ``over-shooting'' with too high a polynomial degree is not computationally useful, it also is not useful to choose $\varepsilon$ much smaller than the desired error. These interactions suggest future work on adaptive approaches to choosing both $\varepsilon$ and anisotropic polynomial degrees. 
The vertical dashed lines in Figure~\ref{fig:Poisson-convergence} show the total number of function evaluations that would be required to evaluate a full tensor pseudospectral approximation of the given degree; as expected, for polynomial degrees larger than one,  \texttt{FTT-projection} requires many orders of magnitude fewer function evaluations than a full tensor approach.


\section{Conclusions}


This paper presents a rigorous construction of the spectral tensor-train (STT) decomposition for multivariate functions. The method aims to mitigate the {curse of dimensionality} for functions with sufficient regularity, by constructing approximations that exploit low tensor rank and that can attain spectral rates of convergence. We present an iterative procedure for decomposing an arbitrary function $f \in L^2_\mu({\bf I})$, yielding a format termed the \textit{functional} tensor-train (FTT) decomposition (to distinguish it from the TT decomposition of discrete tensors). The construction of the FTT decomposition relies on the singular value decomposition of Hilbert-Schmidt kernels in $L^2_\mu({\bf I})$ and on the regularity properties of $f$ (cf.\ Theorem \ref{thm:ftt-approx-conv}). This regularity is inherited by the singular functions or ``cores'' of the decomposition (cf.\ Theorems \ref{thm:ftt-sobolev} and \ref{thm:regularity}).
We then develop error bounds that account for truncation of the FTT decomposition at a given rank and for polynomial approximation of the cores. Collectively, these theoretical results describe the connections between Sobolev regularity of $f$, the dimension of the input space, and approximation rates in terms of tensor rank. 

To implement the spectral tensor-train decomposition numerically, we apply the \texttt{TT-DMRG-cross} sampling algorithm \cite{Savostyanov2011} to a discrete tensor comprising suitably weighted pointwise evaluations of $f$; the definition of this tensor reflects a choice of tensor-product quadrature rule. The user is required to select the polynomial degree of the approximation and the desired relative accuracy. The latter tolerance drives the extent of dimensional interactions described by the approximation and ultimately the number of function evaluations demanded by the algorithm. 
Numerical experiments demonstrate good performance of this approximation. For analytically low-rank functions, empirical results confirm that computational effort (i.e., the number of function evaluations required to achieve a given accuracy) scales linearly with dimension. Even for functions that are not analytically low rank, we observe that the STT approximation significantly outperforms an adaptive sparse grid approach. Recall that the FTT approximation is nonlinear in the sense that it does not prescribe a basis for the separation of the space $L^2_\mu({\bf I})$; instead, it uses the singular functions of $f$, which are optimal. The choice of basis is made when projecting the singular functions onto, for example, a finite-degree polynomial space. This approach also offers the flexibility needed to resolve local features of a function, by clustering the evaluation points close to the feature.


%



Many avenues for further development center on adaptivity. For example, the ordering of the dimensions can have an important impact on the number of function evaluations required to produce an STT approximation; finding an optimal or near-optimal ordering \textit{a priori} or adaptively is a topic of ongoing work. Results from the current work can also pave the way towards a fully adaptive STT decomposition, using the smoothness properties of the singular functions to indicate whether to increase the polynomial degree in each dimension. This will allow a more complete automation of the construction process. Further theoretical developments relating the discrete and functional representations would also be of great interest: for example, describing the relationship between cross-interpolation error and the pointwise approximation of the FTT cores. It would also be useful to extend current results on the convergence of the FTT decomposition to unbounded domains (e.g., $\mathbb{R}^d$) equipped with finite measure. These efforts are left to future work. 

An open-source Python implementation of the STT approximation algorithm including all the numerical examples from this paper is available at \url{http://pypi.python.org/pypi/TensorToolbox/}.

\section*{Acknowledgments}
The authors would like to thank Jan Hesthaven, Alessio Spantini, Florian Augustin, and Patrick Conrad for fruitful discussions on this topic and for providing many useful comments on the paper. We would also like to thank Dmitry Savostyanov for calling our attention to the \texttt{TT-DMRG-cross} algorithm\rvnote*{\#1-1}{, and the anonymous referee who suggested an important improvement to Theorem \ref{thm:ftt-approx-conv}.} D.\ Bigoni acknowledges the financial support of DTU Compute for his initial visit to MIT. D.\ Bigoni and Y.\ Marzouk also acknowledge support from the US Department of Energy, Office of Science, Advanced Scientific Computing Research under award number DE-SC0007099. 

\appendix

\section{H\"older continuity and the Smithies condition}\label{sec:hold-cont-smith}
In Section \ref{sec:regul-ftt-decomp} we use a result by Smithies \cite[Thm.~14]{Smithies1937} to prove the boundedness of the weak derivatives of the cores of the FTT decomposition. The conditions under which Smithies' result holds are as follows:
\smallskip
\begin{definition}[Smithies' integrated H\"older continuity]
\label{def:integratedholder}
  Let $K(s,t)$ be defined for $s,t\in[a,b]$. Without loss of generality, let $a=0$ and $b=\pi$. For $r>0$, let
  \begin{equation}
    K^{(i)}(s,t) = \frac{\partial^i K(s,t)}{\partial s^i} \; , \qquad 0 < i \leq r,
  \end{equation}
  and let $K^{(1)},\ldots,K^{(r-1)}$ exist and be continuous. Let $K^{(r)} \in L^p(s)$ a.e.\ in $t$ for $1 < p \leq 2$. Then integrated H\"older continuity, with either $r>0$ and $\alpha>0$ or $r=0$ and $\alpha>\frac{1}{p}-\frac{1}{2}$, holds for $K$ if and only if there exists an $A>0$ such that:
  \begin{equation}
    \label{eq:int-Holder}
    \int_0^\pi \left\{ \int_0^\pi \left\vert K^{(r)}(s+\theta,t) - K^{(r)}(s-\theta,t) \right\vert^p ds \right\}^{\frac{2}{p}} dt \leq A \vert \theta \vert^{2\alpha} \;.
  \end{equation}
\end{definition}

This definition somewhat difficult to interpret. Furthermore, in the scope of this work, we are interested in the case $r=0$. A simpler, but not equivalent, definition is given in \cite{TOWNSEND2013}:
\smallskip
\begin{definition}[H\"older continuity almost everywhere]
  Let $K(s,t)$ be defined for $s,t\in[a,b]$. $K$ is H\"older continuous a.e.\ with exponent $\alpha>0$ if there exists $C>0$ such that
  \begin{equation}
    \label{eq:Holder-ae}
    \vert K(s+\theta,t) - K(s-\theta,t) \vert \leq C \vert \theta \vert^\alpha
  \end{equation}
  almost everywhere in $t$.
\end{definition}
\smallskip

To clarify the connection between these notions, we will show that:
\begin{proposition}
  H\"older continuity a.e.\ is a sufficient condition for the integrated H\"older continuity given in Definition~\ref{def:integratedholder}.
\end{proposition}
\begin{proof}
  Let $K \in L^p(s)$ for almost all $t$, $1 < p \leq 2$. For $\alpha > \frac{1}{2}$, let $K$ be H\"older continuous a.e.\ in $t$. Then:
\begin{align*}
      \int_0^\pi \left\{ \int_0^\pi \left\vert K^{(r)}(s+\theta,t) - K^{(r)}(s-\theta,t) \right\vert^p ds \right\}^{\frac{2}{p}} dt  
& \leq \int_0^\pi \left\{ \int_0^\pi C^p \left\vert \theta \right\vert^{\alpha p} ds \right\}^{\frac{2}{p}} dt \\
& = C^2 \pi^{\frac{3}{p}} \left\vert \theta \right\vert^{2\alpha} \\
&\leq C^2 \pi^{3} \left\vert \theta \right\vert^{2\alpha} = A \left\vert \theta \right\vert^{2\alpha} \; ,
\end{align*}
where we recognize the bound \eqref{eq:int-Holder} of the Smithies integrated H\"older continuity condition. \hfill \end{proof}

\section{Proofs of auxiliary results for Theorem \ref{thm:ftt-approx-conv}}\label{sec:sobol-cont-kern}

\subsection{Proof of Lemma \ref{lemma:SobolevSemiVsNorm}}
  By definition of Sobolev norm, seminorm and weak derivative $D^{\bf i}$:
  \begin{equation}
    \label{eq:ProofSobolevSemiVSNorm-1}
    \begin{aligned}
      \vert J \vert^2_{I_1\times I_1,\mu,k} &\leq \Vert J \Vert^2_{\mathcal{H}^k_\mu(I_1 \times I_1)} = \sum^k_{\vert {\bf i} \vert = 0} \Vert D^{\bf i} \langle f(x,y), f(\bar{x},y) \rangle_{L^2_\mu(\bar{\bf I})} \Vert^2_{L^2_\mu(I_1\times I_1)} \\
      &= \sum^k_{\vert {\bf i} \vert = 0} \Vert \langle D^{i_1,{\bf 0}}f(x,y), D^{i_2,{\bf 0}}f(\bar{x},y) \rangle_{L^2_\mu(\bar{\bf I})} \Vert^2_{L^2_\mu(I_1\times I_1)} \; ,
    \end{aligned}
  \end{equation}
  where ${\bf i}$ is a two dimensional multi-index.
  Using the Cauchy-Schwarz inequality, it holds that:
  \begin{equation}
    \label{eq:ProofSobolevSemiVSNorm-2}
    \Vert \langle D^{i_1,{\bf 0}}f(x,y), D^{i_2,{\bf 0}}f(\bar{x},y) \rangle_{L^2_\mu(\bar{\bf I})} \Vert^2_{L^2_\mu(I_1\times I_1)} \leq \Vert D^{i_1,{\bf 0}}f(x,y) \Vert^2_{L^2_\mu({\bf I})} \Vert D^{i_2,{\bf 0}}f(x,y) \Vert^2_{L^2_\mu({\bf I})}
  \end{equation}
Now let ${\bf j}$ and ${\bf l}$ be two $d$-dimensional multi-indices. Then \eqref{eq:ProofSobolevSemiVSNorm-1} can be bounded by
  \begin{equation}
    \label{eq:ProofSobolevSemiVSNorm-3}
    \begin{aligned}
      \vert J \vert^2_{I_1\times I_1,\mu,k} &\leq \Vert J \Vert^2_{\mathcal{H}^k_\mu(I_1 \times I_1)} \leq \sum^k_{\vert {\bf i} \vert = 0} \Vert D^{i_1,{\bf 0}}f(x,y) \Vert^2_{L^2_\mu({\bf I})} \Vert D^{i_2,{\bf 0}}f(x,y) \Vert^2_{L^2_\mu({\bf I})} \\
      &\leq \sum^k_{\vert {\bf j} \vert = 0} \sum^k_{\vert {\bf l} \vert = 0} \Vert D^{\bf j}f(x,y) \Vert^2_{L^2_\mu({\bf I})} \Vert D^{\bf l}f(x,y) \Vert^2_{L^2_\mu({\bf I})} \leq \Vert f \Vert^4_{\mathcal{H}^k_\mu({\bf I})} \;.
    \end{aligned}
  \end{equation}
  Since $\Vert J \Vert_{\mathcal{H}^k_\mu(I_1 \times I_1)} \leq \Vert f \Vert^2_{\mathcal{H}^k_\mu({\bf I})} < \infty $ by assumption, then $ J \in \mathcal{H}^k_\mu(I_1 \times I_1) $.
\hfill $\square$
 

\subsection{Proof of Lemma \ref{lemma:SobolevPhiVsF}}
\rvnote*{\#1-1}{We prove the statement for the first dimension; the other dimensions will follow in a similar fashion. For a particular multi-index ${\bf i}=[i_1,\ldots,i_d]$, let ${\bf j}:=[i_2,\ldots,i_d]$. Let also ${\bf I}=I_1\times \cdots \times I_d$ and $\tilde{\bf I}=I_2\times \cdots \times I_d$. Then
\begin{equation}
  \label{eq:SobolevPhiVsF-1}
  \begin{aligned}
    \left\Vert f \right\Vert^2_{\mathcal{H}_\mu^k({\bf I})} &= \sum_{\vert{\bf i}\vert = 0}^k \left\Vert D^{\bf i} f \right\Vert^2_{L_\mu^2({\bf I})} = \sum_{\substack{\vert{\bf i}\vert = 0\\i_1=0}}^k \left\Vert D^{\bf i} f \right\Vert^2_{L_\mu^2({\bf I})} + \sum_{\substack{\vert{\bf i}\vert = 0\\i_1>0}}^k \left\Vert D^{\bf i} f \right\Vert^2_{L_\mu^2({\bf I})} \\
    &= \sum_{\substack{\vert{\bf i}\vert = 0\\i_1=0}}^k \left\Vert D^{\bf i} \sum_{\alpha_1=1}^\infty \sqrt{\lambda_1(\alpha_1)} \gamma_1(x_1;\alpha_1) \varphi_1(\alpha_1;x_2,\ldots,x_d) \right\Vert^2_{L_\mu^2({\bf I})} + \sum_{\substack{\vert{\bf i}\vert = 0\\i_1>0}}^k \left\Vert D^{\bf i} f \right\Vert^2_{L_\mu^2({\bf I})} \\
    &= \sum_{\substack{\vert{\bf i}\vert = 0\\i_1=0}}^k \sum_{\alpha_1=1}^\infty \lambda_1(\alpha_1) \left\Vert \gamma_1(\cdot;\alpha_1) \right\Vert^2_{L_\mu^2(I_1)} \left\Vert D^{\bf j} \varphi_1(\alpha_1;\cdot) \right\Vert^2_{L_\mu^2(\tilde{\bf I})} + \sum_{\substack{\vert{\bf i}\vert = 0\\i_1>0}}^k \left\Vert D^{\bf i} f \right\Vert^2_{L_\mu^2({\bf I})} \\
    &= \sum_{\alpha_1=1}^\infty \left\Vert \left(\sqrt{\lambda_1}\varphi_1\right)(\alpha_1;\cdot) \right\Vert^2_{\mathcal{H}_\mu^k(\tilde{\bf I})} + \sum_{\substack{\vert{\bf i}\vert = 0\\i_1>0}}^k \left\Vert D^{\bf i} f \right\Vert^2_{L_\mu^2({\bf I})} \;,
  \end{aligned}
\end{equation}
where the third equality was obtained using the orthonormality of $\{\gamma_1(\cdot;\alpha_1) \}_{\alpha_1=1}^\infty$ and the H\"older
$(\alpha > 1/2)$ continuity of $f$, as in the proof of Theorem \ref{thm:ftt-sobolev}. \hfill $\square$}

\bibliographystyle{siam}
\bibliography{biblio}

\end{document}